\documentclass[11pt,a4paper]{amsart}

\usepackage{amsmath}
\usepackage{amsfonts}
\usepackage{graphicx}
\usepackage{framed}
\usepackage{amssymb}
\usepackage{esvect}
\usepackage{xcolor}

\makeatletter
\@namedef{subjclassname@2020}{\textup{2020} Mathematics Subject Classification}
\makeatother

\usepackage{etex}
\usepackage{latexsym}
\usepackage[english]{babel}
\usepackage[utf8x]{inputenc}

\def \N {\mathbb{N}}
\def \R {\mathbb{R}}

\def \de {\partial}
\def \LL {\mathcal{L}}

\def \BB {\mathcal{B}}
\def \DD {\mathcal{D}}

\usepackage{amsthm}
\theoremstyle{definition}
\newtheorem{definition}{Definition}[section]

\newtheorem{remark}[definition]{Remark}

\theoremstyle{plain}

\newtheorem{theorem}[definition]{Theorem}

\newtheorem{lemma}[definition]{Lemma}

\newtheorem{problem}[definition]{Problem}

\numberwithin{equation}{section}
\begin{document}
 \begin{abstract}
 We consider the first Dirichlet eigenvalue problem
 for a mixed local/nonlocal elliptic operator and we establish
 a quantitative Faber-Krahn inequality. More precisely, we show that balls
 minimize the first eigenvalue among sets of given volume and we provide
 a stability result for sets that almost attain the minimum.
 \end{abstract}
 \title[Faber-Krahn for
 mixed operators]{A Faber-Krahn inequality \\ for mixed local and nonlocal operators}
 
 \author[S.\,Biagi]{Stefano Biagi}
 \author[S.\,Dipierro]{Serena Dipierro}
 \author[E.\,Valdinoci]{Enrico Valdinoci}
 \author[E.\,Vecchi]{Eugenio Vecchi}
 
 \address[S.\,Biagi]{Dipartimento di Matematica
 \newline\indent Politecnico di Milano \newline\indent
 Via Bonardi 9, 20133 Milano, Italy}
 \email{stefano.biagi@polimi.it}
 
 \address[S.\,Dipierro]{Department of Mathematics and Statistics
 \newline\indent University of Western Australia \newline\indent
 35 Stirling Highway, WA 6009 Crawley, Australia}
 \email{serena.dipierro@uwa.edu.au}
 
 \address[E.\,Valdinoci]{Department of Mathematics and Statistics
 \newline\indent University of Western Australia \newline\indent
 35 Stirling Highway, WA 6009 Crawley, Australia}
 \email{enrico.valdinoci@uwa.edu.au}
 
 \address[E.\,Vecchi]{Dipartimento di Matematica
 \newline\indent Universit\`a degli Studi di Bologna \newline\indent
Piazza di Porta San Donato 5, 40126 Bologna, Italy}
 \email{eugenio.vecchi2@unibo.it}



\thanks{The authors are members of INdAM. S. Biagi
is partially supported by the INdAM-GNAMPA project 
\emph{Metodi topologici per problemi al contorno associati a certe 
classi di equazioni alle derivate parziali}.
S. Dipierro and E. Valdinoci are members of AustMS.
S. Dipierro is supported by
the Australian Research Council DECRA DE180100957
``PDEs, free boundaries and applications''.
E. Valdinoci is supported by the Australian Laureate Fellowship
FL190100081
``Minimal surfaces, free boundaries and partial differential equations''.
E. Vecchi is partially supported
by the INdAM-GNAMPA project 
\emph{Convergenze variazionali per funzionali e operatori dipendenti da campi vettoriali}.}

 \date{\today}
 
 \maketitle
 
 \section{Introduction}
 
 At the end of the XIX century, relying on explicit calculations on suitable domains, John William Strutt,
 3rd Baron Rayleigh, conjectured that the ball is the minimizer of the first Dirichlet eigenvalue among
 the domains of a given volume, see~\cite{RAY}.
The confirmation of this conjecture entails a number of interesting physical consequences, such as:
\begin{itemize}
\item among all drums of a given surface, the circular drum produces the lowest voice,
\item among all the regions of a given volume
with the boundary maintained at a constant (say, zero) temperature, the one which dissipates heat at the slowest possible rate is the sphere.
\end{itemize}
 Also, the statement with equal volume constraints gives as a byproduct the one with equal perimeter 
 constraint (thanks to the scaling property of the first ei\-gen\-va\-lue and the isoperimetric 
 inequality).
 In this sense, the first attempt to prove
 Lord Rayleigh's conjecture dates back to 1918, when Richard Courant established the above claim with equal 
 perimeter constraint,
 see~\cite{COURANT}. Then, using rearrangement methods and the variational characterization of eigenvalues, 
 the original conjecture with volume constraint was established 
 independently by Georg Faber and Edgar Krahn,
 see~\cite{FABER, KRAHN1, KRAHN2}. See also~\cite[Chapter~2]{HENROT} and \cite{Kes}. We refer to 
 \cite{CupVec} for similar results in the context of composite membranes.

 Given that balls are actually established to be the unique minimizers for the first eigenvalue under 
 volume constraint
 (hence if the first eigenvalue is equal to that of the corresponding ball, then the domain must 
 necessarily be a ball), an intense research activity focused on quantitative versions of the Faber-Krahn 
 inequality: roughly speaking,
 if the eigenvalue is ``close to the one of the ball'', can one deduce 
 that the domain is also ``close to a ball''?
 Classical results in this direction have been obtained by Wolfhard Hansen 
 and Nikolai  Nadirashvili in~\cite{HANSEN}
 and Antonios Melas in~\cite{Melas}, and sharp bounds in terms of the so called Fraenkel asymmetry have 
 been obtained recently by Lorenzo Brasco, Guido De Philippis and Bozhidar Velichkov in~\cite{BRASCO}.
 See also \cite{Avila} for some stability results in space forms.
 \medskip

 The goal of this paper is to obtain a Faber-Krahn inequality and a quantitative
 version of it for an elliptic operator of mixed order. More specifically, for the sake of simplicity,
 we will focus on operators obtained by the superposition
 of a classical and a fractional Laplacian, namely of operators of the form
 $$ \LL :=-\Delta +(-\Delta)^s ,$$
 with~$s\in(0,1)$ and
 $$ (-\Delta)^s u(x):= \frac12\int_{\R^n}\frac{2u(x)-u(x+y)-u(x-y)}{|y|^{n+2s}}\,d y.$$
 Operators of this type present interesting mathematical
 questions, especially due to the lack of scale invariance and in view of the combination
 of local and nonlocal behaviors, see~\cite{JAKO1, JAKO2, BARLES, LLAVE, BISWAP, CIOMAGA, BAR1, CAFFA,
 BAR2, ROSOT, CAB, TESO, VESPRI, PATA, BDVV2, BDVV, PRO, CABRE, ABA, BSM, SilvaSalort, Kinnunen}.
 Moreover, they possess a concrete interest in applications
 since they model diffusion patterns with different time scales
 (loosely speaking,
 the higher order operator leading the diffusion for small times
 and
 the lower order operator becoming predominant for large times)
 and they arise in bi-modal power-law distribution processes,
 see~\cite{PAG}. Further applications arise in the theory of optimal searching strategies, biomathematics
 and animal foraging, see~\cite{PROIETTI, NICHE} and the references therein.
 See also~\cite{CAST} for further applications.
 \medskip

 In our setting, 
 given a bounded open subset~$\Omega$ of~$\R^n$, we consider
 the first Dirichlet eigenvalue~$
 \lambda_\LL(\Omega)$ (see Section~\ref{sec.firstEig} for a detailed presentation) and we characterize
 the optimal set by the following result:
 
 \begin{theorem}[Faber-Krahn inequality for $\lambda_\LL(\Omega)$] \label{thm.FK}
  Let~$\Omega\subseteq\R^n$ be a 
   bounded open set with boundary $\de\Omega$ of class $C^1$.
   Let $m:= |\Omega|\in (0,\infty)$, and let $B^{(m)}$ be any Euclidean ball with 
   volume
   $m$. Then,
   \begin{equation} \label{eq.FK}
    \lambda_\LL(\Omega)\geq \lambda_\LL (B^{(m)}).
   \end{equation}
   Moreover, if the equality holds in \eqref{eq.FK}, then 
   $\Omega$ is a ball.
  \end{theorem}
  
  A related Faber-Krahn inequality has been recently obtained for radially symmetric, nonnegative and continuous kernels with compact support in~\cite{GOE}.
  With this respect, the case treated here of singular kernels seems
  to be new to the best of our knowledge. Additionally, 
 and more importantly,
  we establish a stability result for inequality~\eqref{eq.FK}:
  
 \begin{theorem}[Quantitative Faber-Krahn inequality for $\lambda_\LL(\Omega)$]  \label{thm:mainFK} Let~$s\in(0,1/2)$.
  Let $\Omega \subset \mathbb{R}^n$ be an open, bounded 
  and uniformly convex set with boundary $\partial \Omega$ of class $C^3$.

  Then, there exists $\varepsilon_0 > 0$ with the following property:
   if $B\subseteq \R^n$ is a ball with~$|B|=|\Omega|$ and
   \begin{equation}\label{eq:AssumptionFK}
   \lambda_{\LL}(\Omega) \leq (1+\varepsilon ) \lambda_{\LL}(B), 
   \quad \textrm{for some } 0< \varepsilon <\varepsilon_0,
   \end{equation}
   then there exist
   two balls $B^{(1)}$, $B^{(2)} \subset \mathbb{R}^n$ such that 
   $B^{(1)} \subseteq \Omega \subseteq B^{(2)}$ and
   \begin{equation}\label{eq:Goal1}
   |B^{(1)}| \geq \left(1-C \varepsilon^{\frac1{2n}}\right)|\Omega|, 
\end{equation}
and
\begin{equation}\label{eq:Goal2}   
   |\Omega| \geq \left(1-C \varepsilon^{\frac1{n(n+1)}}\right)|B^{(2)}|,
   \end{equation}
where $C>0$ is a structural constant.
  \end{theorem}
 The proof of Theorem~\ref{thm.FK} is based on the definition of principal eigenvalue of~$\LL$ and
 on the Polya-Szeg\"o inequality. 

 The proof of Theorem~\ref{thm:mainFK} is more involved and requires some delicate geometric arguments. 
 For this, a careful analysis of the superlevel sets of the principal eigenfunction is needed, combined
 with the use of a Bonnesen-type inequality, which is a strengthening of the classical isoperimetric 
 inequality.

 We stress that, while in Theorem~\ref{thm.FK} we only require the boundary of~$\Omega$ to be~$C^1$, in
 Theorem~\ref{thm:mainFK} we need to ask that the boundary of~$\Omega$ is of class~$C^3$. This is due to  
 the
 fact that, in order to prove Theorem~\ref{thm:mainFK}, we have to use the Faber-Krahn inequality
 in~\eqref{eq.FK} for the superlevel sets of the principal eigenfunction and we are able to prove that 
 these sets
 are convex and with boundary of class~$C^1$ under the condition that~$\Omega$ has a~$C^3$-boundary,
 see the forthcoming Lemma~\ref{lem:convexity}. 
 \medskip

 The rest of this paper is organized as follows. In Section~\ref{sec.main} we introduce the basic notation  
 and the
 setting in which we work, and we provide some regularity results 
 and an Hopf-type lemma for the operator~$\LL$.
 In Section~\ref{sec.firstEig} we introduce the notion of principal Dirichlet eigenvalue for~$\LL$ and we 
 give
 some regularity results for the associated eigenfunction. 
 Section~\ref{sec.firstEig} also contains a result (namely, Lemma~\ref{lem:convexity})
 that proves the convexity of the superlevel sets of the first eigenfunction
 near the boundary of a convex domain: this
 result, based on a detailed use of the 
 Inverse Function Theorem, highlights
 an interesting technical difference
 with respect to the classical case where
 one can exploit the convexity of all the level sets of the Dirichlet principal
 eigenfunction, due to its concavity as a function (which is unknown
 in the fractional case, see~\cite{KASI}).
 
 Section~\ref{pr:sec} contains the proofs
 of Theorems~\ref{thm.FK} and~\ref{thm:mainFK}.
 The paper finishes with a couple of appendices: in the first one, we discuss the optimality of a geometric 
 lemma needed in the proof of Theorem~\ref{thm:mainFK}, while in the second we prove Theorem  
 \ref{thm:solvabLL}.
 
 \section{Basic notions and preliminary results} \label{sec.main}
 In this section we properly introduce the relevant definitions and notation
 which shall be used throughout the rest of the paper. Moreover, 
 we review/establish some regularity results concerning our operator $\LL
 = -\Delta+(-\Delta)^s$.
 \medskip
 
 Let $\Omega\subseteq\R^n$ be a bounded open set with boundary
 $\de\Omega$ of class $C^{1}$. 
 We then consider the space $\mathbb{X}(\Omega)$ defined as follows:
 $$\mathbb{X}(\Omega) :=
 \big\{u\in H^1(\R^n):\,\text{$u\equiv 0$ a.e.\,on $\R^n\setminus\Omega$}\big\}.$$
 In view of the regularity of $\de\Omega$, it is well-known that
 $\mathbb{X}(\Omega)$ can be naturally i\-den\-ti\-fied with 
 $H_0^1(\Omega)$; more precisely, we have (see, e.g., \cite[Prop.\,9.18]{Brezis})
 \begin{equation} \label{eq.indetifyXHzero}
 \text{$u\in\mathbb{X}(\Omega)\,\,\Longrightarrow\,\,u|_\Omega\in H_0^1(\Omega)$}
 \quad\text{and}\quad
 \text{$u\in H_0^1(\Omega)\,\,\Longrightarrow\,\,u\cdot\mathbf{1}_\Omega\in \mathbb{X}(\Omega)$},
 \end{equation}
 where $\mathbf{1}_\Omega$ denotes the indicator function of $\Omega$.
 Throughout what follows, 
 we tacitly identify a function $u\in H_0^1(\Omega)$ with its
 ``zero-extension'' $\widehat{u} = u\cdot\mathbf{1}_{\Omega}\in \mathbb{X}(\Omega)$.
 
 We then observe that, as a consequence of \eqref{eq.indetifyXHzero}, 
 the set
 $\mathbb{X}(\Omega)$ is endowed with a
 structure of real Hilbert space by the scalar product 
 \begin{equation*}
 \langle u, v\rangle_{\mathbb{X}(\Omega)}
  := \int_\Omega \langle \nabla u, \nabla v\rangle\,d x.
  \end{equation*}
 The norm associated with this scalar product is
 \begin{equation*}
\|u\|_{\mathbb{X}(\Omega)} := \big(\int_\Omega|\nabla u|^2\,d x\big)^{1/2},
  \end{equation*}
 and the (linear) map $\mathcal{E}_0:H_0^1(\Omega)\to\mathbb{X}(\Omega)$ defined by
 $$\mathcal{E}_0(u) := u\cdot\mathbf{1}_\Omega$$
 turns out to be a bijective isometry between $H_0^1(\Omega)$ and $\mathbb{X}(\Omega)$.
 \medskip
 
 Let~$s\in (0,1)$. On the space
 $\mathbb{X}(\Omega)$, we consider the bilinear form
 $$\mathcal{B}_{\Omega,s}(u,v) :=
 \int_{\Omega}\langle \nabla u,\nabla v\rangle\,d x
 + \iint_{\R^{2n}}\frac{(u(x)-u(y))(v(x)-v(y))}{|x-y|^{n+2s}}\, dx\, dy;$$
 moreover, for every $u\in \mathbb{X}(\Omega)$ we define
 \begin{equation} \label{eq.defquadD}
  \mathcal{D}_{\Omega,s}(u) := \mathcal{B}_{\Omega,s}(u,u).
 \end{equation}
 \begin{remark} \label{rem.welldef}
  We explicitly notice that $\mathcal{B}_{\Omega,s}$ is well-defined
  on $\mathbb{X}(\Omega)$ in view of the following facts:
  given any $u,v\in \mathbb{X}(\Omega)$, by H\"older's inequality we have
  \begin{equation} \label{eq.estimHolder}
    \iint_{\R^{2n}}\frac{|u(x)-u(y)|\cdot |v(x)-v(y)|}{|x-y|^{n+2s}}\, d x\, dy
   \leq [u]_s\cdot [v]_s,
  \end{equation}
  where we have used the notation
  $$[f]_s := \bigg(\iint_{\R^{2n}}\frac{|f(x)-f(y)|^2}{|x-y|^{n+2s}}\, d x\, dy\bigg)^{1/2}
  \qquad \text{for all~$f\in H^1(\R^n)$}.$$
  Furthermore, for every  $f\in H^1(\R^n)$ one has
  (see, e.g., \cite[Proposition 2.2]{guida})
  \begin{equation} \label{eq.estimsW1p}
   [f]_s \leq c_{n,s}\,\|f\|_{H^1(\R^n)}.
   \end{equation}
  Gathering together \eqref{eq.estimHolder} and \eqref{eq.estimsW1p}, we then get
  $$
  |\mathcal{B}_{\Omega,s}(u,v)|\leq c_{n,s}^2\,\|u\|_{H^1(\R^n)}\cdot \|v\|_{H^1(\R^n)}
  < \infty.$$
 \end{remark}
 
 Using the bilinear form $\BB_{\Omega,s}$, we can give the following definition.
 \begin{definition} \label{def.Weak}
 Let $f\in L^2(\Omega)$. We say that a function $u:\R^n\to\R$
 is a \emph{weak so\-lu\-tion} of the $\LL$-Dirichlet problem
 $$(\mathrm{D})_{f}\qquad\quad
 \begin{cases}
 \LL u = f & \text{in $\Omega$}, \\
 u \equiv 0 & \text{in $\R^n\setminus\Omega$},
 \end{cases}$$
 if it satisfies the following properties:
 \begin{enumerate}
  \item $u\in\mathbb{X}(\Omega)$;
  \item for every test function $\varphi\in C_0^\infty(\Omega)$, one has
 $$\mathcal{B}_{\Omega,s}(u,\varphi) = \int_{\Omega}f\varphi\, dx.$$
 \end{enumerate}
 \end{definition}
 
 \begin{remark} \label{rem.density}
  Let $f\in L^2(\Omega)$ and $u\in\mathbb{X}(\Omega)$. Since
  $C_0^\infty(\Omega)$ is \emph{dense} in $\mathbb{X}(\Omega)$, we see that $u$ is a weak
  solution of $(\mathrm{D})_{f}$ \emph{if and only if}
  $$\mathcal{B}_{\Omega,s}(u,v) = \int_{\Omega}f\,v\, dx\qquad\text{for every $v\in\mathbb{X}(\Omega)$}.$$
 \end{remark}
 
 Then, by applying the Lax-Milgram Theorem to the bilinear form $\BB_{\Omega,s}$,
one can prove the following existence result (see, e.g., \cite[Theorem.\,1.1]{BDVV}).

 \begin{theorem} \label{thm:existenceBasic}
  For every $f\in L^2(\Omega)$, there exists a \emph{unique weak solution}~$u_f\in \mathbb{X}(\Omega)$ of $(\mathrm{D})_{f}$, further
  satisfying the `a-priori' estimate
  $$\|u_f\|_{\mathbb{X}(\Omega)} \leq \mathbf{c}_0\,\|f\|_{L^2(\Omega)}.$$
  Here, $\mathbf{c}_0 > 0$ is a constant \emph{independent of $f$}.  
 \end{theorem}
 
 Since we aim at proving some \emph{global regularity results} for $u_f$,
 it is convenient to fix also the definition of \emph{classical solution} of problem
 $(\mathrm{D})_{f}$. For this, we recall the notation~$C_b(\R^n):=C(\R^n)\cap L^\infty(\R^n)$.
 
 \begin{definition} \label{def:classicalsol} 
 Let $f:\Omega\to\R$. We say that a function $u:\R^n\to\R$
 is a \emph{classical solution} of 
 $(\mathrm{D})_{f}$ if it satisfies the following properties:
 \begin{enumerate}
  \item $u\in C_b(\R^n)\cap C^2(\Omega)$;
  \item $u\equiv 0$ pointwise in $\R^n\setminus\Omega$;
  \item $\LL u(x) = f(x)$ pointwise for every $x\in\Omega$.
 \end{enumerate}
 \end{definition}
 
 \begin{remark} \label{rem:weakclassicalcfr}
 (1)\,\,We explicitly observe that, if
  $u\in C_b(\R^n)\cap C^{2}(\Omega)$
  it is possible to
  compute $\LL u(x)$ pointwise for every $x\in\Omega$. Indeed, we have
  $$(-\Delta)^s u(x) = -\frac{1}{2}\int_{\R^n}\frac{u(x+z)+u(x-z)-2u(x)}{|z|^{n+2s}}\, d z,$$
  and the regularity of $u$ ensures that the ``second-order'' different quotient
 $$z\ni\R^n\mapsto \frac{u(x+z)+u(x-z)-2u(x)}{|z|^{n+2s}}$$
 is in $L^1(\R^n)$. To be more precise, for every $x\in\R^n$ we have
 \begin{align*}
  & \frac{|u(x+z)+u(x-z)-2u(x)|}{|z|^{n+2s}} \\
  & \quad \leq 
  c_n\bigg(
  \frac{\|u\|_{C^2(B(x,\rho_x))}}{|z|^{n+2s-2}}\cdot\mathbf{1}_{B(0,\rho_x)} + 
  \frac{\|u\|_{L^\infty(\R^n)}}{|z|^{n+2s}}\cdot\mathbf{1}_{\R^n\setminus B(0,\rho_x)}\bigg),
 \end{align*}
 where $c_n > 0$ is a suitable constant and $\rho_x > 0$ is such that
 $B(x,\rho_x)\Subset\Omega$.
 \medskip
 
 (2)\,\,Assume that $f\in L^2(\Omega)$, and let $u_f\in\mathbb{X}(\Omega)$ be the (unique)
  weak solution of~$(\mathrm{D})_{f}$, according to Theorem~\ref{thm:existenceBasic}.
  If we further assume that $u_f\in C_b(\R^n)\cap C^2(\Omega)$,
  we can compute
  $$\LL u_f(x) = -\Delta u_f(x)+(-\Delta)^s u_f(x)\quad\text{pointwise for every
  $x\in\Omega$}.$$
  Then, a standard integration-by-parts argument shows that $u_f$ is also a classical
  solution of $(\mathrm{D})_{f}$. Conversely, if $u\in C_b(\R^n)\cap C^2(\Omega)$ 
  is a classical solution
  of $(\mathrm{D})_{f}$ such that $u\in H^1(\R^n)$, then $u$ is also
  a weak solution of $(\mathrm{D})_{f}$. 
 \end{remark}
 
 The first regularity result that we aim to prove is
 a \emph{global $C^{1,\alpha}$-re\-gu\-la\-rity theorem}, which relies
 on the $W^{2,p}$-theory for $\LL$ developed by
 Bensoussan and Lions \cite{BeLi}.
 
 \begin{theorem} \label{thm:W2pBL}
  Let $f\in L^\infty(\Omega)$ and let $u_f\in\mathbb{X}(\Omega)$ denote the unique
  weak solution of $(\mathrm{D})_{f}$. Moreover,
  let us assume that $\de\Omega$ is of class $C^{1,1}$. Then,
  \begin{equation} \label{eq:ufConeHolder}
   u_f\in C^{1,\beta}(\overline{\Omega})\quad\text{for some $\beta\in(0,1)$}.
  \end{equation}
 \end{theorem}
 
 \begin{proof} Note that
  $f\in L^p(\Omega)$
  for every $p\geq 2$.
  We\footnote{For instance, in the most delicate case~$s\in[ 1/2,1)$, one can use the setting in~\cite{GaMe}
with~$j(x,\xi): = \xi$,
$m(x,\xi):= 1$, $\pi(d\xi) := |\xi|^{-n-2s}\, d\xi$, $
\overline{j}(\xi): = \xi$, $
\gamma := 2s+\epsilon$ (with~$\epsilon>0$ conveniently small),
$\theta: = 0$, $
\gamma_1 := \gamma$ and~$
\lambda_1(\xi) := |\xi|^\gamma$.

Alternatively to~\cite{GaMe}, one can use~\cite[Theorem 3.2.3]{BeLi}.} utilize~\cite[Theorem 3.1.22]{GaMe}
to obtain $W^{2,p}$-regularity for some~$p>n$.
By combining this with the Sobolev embedding theorem, we obtain \eqref{eq:ufConeHolder}. \end{proof}
  
Furthermore, we recall that the \emph{$C^{2,\alpha}$-regularity of $u_f$} when $f\in C^{\alpha}(\overline{\Omega})$ can be deduced by applying \cite[Theorem 3.1.12]{GaMe}. We refer to Appendix \ref{AppC2alpha} for an explicit proof.
 
 \begin{theorem} \label{thm:solvabLL}
  Let~$s\in (0,1/2)$ and~$\alpha\in (0,1)$. 
Suppose that $\de\Omega$ is of class $C^{2,\alpha}$.
  If $f\in C^{\alpha}(\overline{\Omega})$ and if $u_f\in\mathbb{X}(\Omega)$
  denotes the unique weak solution of $(\mathrm{D})_{f}$
  \emph{(}according to Theorem \ref{thm:existenceBasic}\emph{)}, then
  $$u_f\in C_b(\R^n)\cap C^{2,\alpha}(\overline{\Omega}).$$
  In particular, $u_f$ is a \emph{classical solution} of  $(\mathrm{D})_{f}$.
 \end{theorem}
 
We also recall that
a regularity result up to the boundary for eigenfunctions
of mixed operators has been recently obtained in~\cite{DELP} for radially symmetric, nonnegative and continuous kernels with compact support.
 
 We close this section by stating, for future reference,
 the following Hopf-type lem\-ma for our operator $\LL$. Similar statements for this type of operators
 have been proved
 in~\cite[Theorem\,3.1.5]{GaMe}, but with the additional assumption that~$s\in(0,1/2)$.
  
 \begin{theorem} \label{thm:Hopf} Let~$c_0\in\R$ and~$\overline\varepsilon\in(0,1)$.
 Let $u\in C_b(\R^n)\cap C^2(\overline{\Omega})$ be such that
 \begin{equation}\label{LLIPO}
 \LL u \geq 0 \quad \text{pointwise in $\Omega$}.
 \end{equation}
 Let $\xi\in\de\Omega$. Assume that~$u=c_0$ on~$B_{\overline\varepsilon}(\xi)\cap\partial\Omega$ 
 and that $u\ge c_0$ in~$\R^n$.
 Suppose also that~$u\not\equiv c_0$ and that~$B_{\overline\varepsilon}(\xi)\cap\de\Omega$ is 
 of class~$C^2$. Finally, let $\nu$ be the outer unit normal to $\de\Omega$ at $\xi$. Then,
\begin{equation}\label{LLIPO2}
\de_{\nu}u(\xi) < 0.
\end{equation}
 \end{theorem}
 
 \begin{proof} Without loss of generality, we suppose that~$c_0=0$ and that~$\xi$ coincides with the origin. In particular,
 \begin{equation}\label{aoa232843vb480}
 u(0)=0.\end{equation}
 Also, by the regularity of~$B_{\overline\varepsilon}(\xi)\cap\partial\Omega$, we suppose that there exists~$\rho_0>0$ such that~$B_{\rho_0}(\rho_0 e_n)\subseteq\Omega$ touches
 the boundary of~$\Omega$ at the origin. 
 We remark that, by assumption, there exists a point~$p\in\R^n$ such that~$u(p)>0$, and therefore, by continuity,
there exists a ball~$B\subset\R^n$ such that~$u>0$ in~$B$.
In light of this
and~\eqref{aoa232843vb480}, we can find~$\varepsilon_0\in(0,\overline\varepsilon/4)$ sufficiently small
such that 
$$\text{$B\subset \R^n\setminus B_{\varepsilon}$ for all~$\varepsilon\in(0,\varepsilon_0]$}.$$
Now, we claim that there exists~$\rho\in(0,\rho_0]$ such that
 \begin{equation}\label{poe37567ghgvhvgw}
 (-\Delta)^s u\le0 \quad {\mbox{ in }} B_\rho(\rho e_n). 
 \end{equation}
To prove this, we suppose by contradiction that for every~$k\in\N$ there exists
 $$p_k\in B_{1/k}( e_n/k)\subseteq\Omega$$
such that~$(-\Delta)^s u(p_k)>0$. This implies that, for every~$\varepsilon\in(0,\varepsilon_0]$,
\begin{equation}\begin{split}\label{w24v58bv587}
0&\le\, \liminf_{k\to+\infty} \int_{\R^n}\frac{2u(p_k)-u(p_k+y)-u(p_k-y)}{|y|^{n+2s}}\, d y\\
&=\, \liminf_{k\to+\infty} 
\Bigg(\int_{B_{\varepsilon}}\frac{2u(p_k)-u(p_k+y)-u(p_k-y)}{|y|^{n+2s}}\, d y\\ & 
\qquad \qquad\qquad
+ \int_{\R^n\setminus B_{\varepsilon}}\frac{2u(p_k)-u(p_k+y)-u(p_k-y)}{|y|^{n+2s}}\, d y\Bigg).
\end{split}
\end{equation}
Notice that~$p_k\to0$ as~$k\to+\infty$, and therefore, by the Dominated Con\-ver\-gen\-ce Theorem, one sees that
\begin{equation*}
\begin{split}&
\lim_{k\to+\infty}\int_{\R^n\setminus B_{\varepsilon}}\frac{2u(p_k)-u(p_k+y)-u(p_k-y)}{|y|^{n+2s}}\, d y
\\&\qquad=\int_{\R^n\setminus B_{\varepsilon}}\frac{2u(0)-u(y)-u(-y)}{|y|^{n+2s}}\, d y
=-\int_{\R^n\setminus B_{\varepsilon}}\frac{u(y)+u(-y)}{|y|^{n+2s}}\, d y\\&\qquad\le
-\int_{B}\frac{u(y)+u(-y)}{|y|^{n+2s}}\, d y\le-\int_{B}\frac{u(y)}{|y|^{n+2s}}\, d y\le -c,
\end{split}
\end{equation*}
for some~$c>0$ independent of~$\varepsilon$.
Plugging this information into~\eqref{w24v58bv587}, we obtain that, for every~$\varepsilon\in(0,\varepsilon_0]$,
\begin{equation}\label{si39v3875885}
c\le \liminf_{k\to+\infty} \int_{B_{\varepsilon}}\frac{2u(p_k)-u(p_k+y)-u(p_k-y)}{|y|^{n+2s}}\, d y.
\end{equation}
Now we claim that there exists~$\widetilde u\in C^{1,1}(B_{2\varepsilon})$ such that
\begin{equation}\label{EXTENS}
{\mbox{$\widetilde u=u$ in~$B_{2\varepsilon}\cap\Omega$ and
$\widetilde u\le u$ in~$B_{2\varepsilon}\setminus\Omega$.}}
\end{equation}
To this end, we consider a~$C^2$-diffeomorphism~$\Phi$ such that
\begin{itemize}
 \item[(a)] $B_{2\varepsilon}\subseteq \Phi(B_{C_0\varepsilon})$ for a suitable~$C_0>0$;
 \item[(b)] $B_{2\varepsilon}\cap\Omega\subseteq
\Phi(B_{C_0\varepsilon}\cap\{x_n>0\})\subseteq\Omega$;
 \item[(c)] $B_{2\varepsilon}\setminus\Omega\subseteq
\Phi(B_{C_0\varepsilon}\setminus\{x_n>0\})\subseteq\R^n\setminus\Omega$.
\end{itemize}
 Possibly reducing~$\varepsilon$ we can also assume that 
 $$\Phi(B_{C_0\varepsilon})\subseteq B_{\overline\varepsilon/2}.$$
For all~$x\in B_{C_0\varepsilon}$ we define~$U(x):=u(\Phi(x))$
and we notice that $U$ is a $C^{1,1}$ function
in~$B_{C_0\varepsilon}\cap\{x_n\ge0\}$
and~$U=0$ along~$\{x_n=0\}$.
Thus, we define
$$ \widetilde{U}(x):=\begin{cases}
U(x',x_n) & {\mbox{ if }}x_n>0,\\
-U(x',-x_n) & {\mbox{ if }}x_n\le0.
\end{cases}$$
We observe that~$\widetilde{U}$ is continuous across~$\{x_n=0\}$ and that
$$\nabla \widetilde{U}(x)=\begin{cases}
\nabla U(x',x_n) & {\mbox{ if }}x_n>0,\\
\nabla U(x',-x_n) & {\mbox{ if }}x_n<0.
\end{cases}$$
This gives that~$\widetilde U$ is a~$C^1$ function in~$B_{C_0\varepsilon}$. Moreover, we
claim that
\begin{equation}\label{C11tile}
{\mbox{$\widetilde U$ is a~$C^{1,1}$ function in~$B_{C_0\varepsilon}$.}}
\end{equation}
To check this, it suffices to consider~$x=(x',x_n)$ and~$y=(y',y_n)$ with~$x_n>0>y_n$ and observe that
\begin{eqnarray*}
|\nabla \widetilde{U}(x)-\nabla \widetilde{U}(y)|&=&|\nabla U(x',x_n)-
\nabla U(y',-y_n)|\\&\le&
|\nabla U(x',x_n)-
\nabla U(x',0)|
+|\nabla U(x',0)-
\nabla U(y',0)|\\&&\qquad\qquad
+|\nabla U(y',0)-
\nabla U(y',-y_n)|\\&\le&
\|U\|_{C^{1,1}(B_{C_0\varepsilon}\cap\{x_n\ge0\})}\Big(
x_n+|x'-y'|-y_n
\Big)
\\&\le&
2\|U\|_{C^{1,1}(B_{C_0\varepsilon}\cap\{x_n\ge0\})}|x-y|,
\end{eqnarray*}
thus establishing~\eqref{C11tile}.

Hence, for every~$x\in B_{2\varepsilon}$ we define~$\widetilde u(x):=\widetilde U(\Phi^{-1}(x))$
and we infer from~\eqref{C11tile} that~$\widetilde u\in C^{1,1}(B_{2\varepsilon})$.
Furthermore, if~$x\in B_{2\varepsilon}\cap\Omega$, then 
 $$\Phi^{-1}(x)\in B_{C_0\varepsilon}\cap\{x_n>0\}$$
and, as a result, we have that~$\widetilde u(x)=\widetilde U(\Phi^{-1}(x))
= U(\Phi^{-1}(x))=u(x)$. If instead~$x\in B_{2\varepsilon}\setminus\Omega$, then~$\Phi^{-1}(x)\in B_{C_0\varepsilon}\cap\{x_n\le0\}$.
Hence, letting 
$$(y',y_n):=\Phi^{-1}(x)$$ 
and using that~$u\ge 0$,
we get in this case that
\begin{align*} 
  \widetilde u(x) & = \widetilde U(\Phi^{-1}(x))=
\widetilde U(y',y_n)=-U(y',-y_n)\le0 \\
 & \le U(y',y_n)=U(\Phi^{-1}(x))=u(x).
\end{align*}
These observations complete the proof of~\eqref{EXTENS}.
\medskip

By~\eqref{EXTENS} it follows that
$$ 2u(p_k)-u(p_k+y)-u(p_k-y)\le
2\widetilde u(p_k)-\widetilde u(p_k+y)-\widetilde u(p_k-y).$$
Hence, if~$y\in B_\varepsilon$,
$$ 2u(p_k)-u(p_k+y)-u(p_k-y)\le\|\widetilde u\|_{C^{1,1}(B_{2\varepsilon})}|y|^2$$
and accordingly
\begin{eqnarray*}
\int_{B_{\varepsilon}}\frac{2u(p_k)-u(p_k+y)-u(p_k-y)}{|y|^{n+2s}}\, d y&\le&
\int_{B_{\varepsilon}}
\|\widetilde u\|_{C^{1,1}(B_{2\varepsilon})}|y|^{2-n-2s}
\, d y\\&\le& C\,\|\widetilde u\|_{C^{1,1}(B_{2\varepsilon})}\varepsilon^{2-2s}.
\end{eqnarray*}
This and~\eqref{si39v3875885} entail that~$c\le C\,\|\widetilde u\|_{C^{1,1}(B_{2\varepsilon})}
\varepsilon^{2-2s}$,
and we thereby obtain a contradiction
by choosing $\varepsilon$ conveniently small.
This completes the proof of~\eqref{poe37567ghgvhvgw}.

{F}rom~\eqref{LLIPO} and~\eqref{poe37567ghgvhvgw} we deduce that there exists~$\rho>0$ such that
$$ \text{$0\le \LL u =-\Delta u +(-\Delta)^s u\le -\Delta u$ in~$B_\rho(\rho e_n)$}.$$
Moreover, we have~$u>0$ in~$B_\rho(\rho e_n)$. Indeed, if there exists a point~$q\in B_\rho(\rho e_n)$
such that~$u(q)=0$,
then by the Maximum Principle (see~\cite[Theorem 1.4]{BDVV}) we would have that~$u\equiv0$ in~$\R^n$,
which would contradict our hypothesis. 
As a con\-se\-quen\-ce, we are in the position of applying the Hopf Lemma 
for the classical Laplacian (see e.g.~\cite[Lemma 3.4]{GT})
thus obtaining the desired result in~\eqref{LLIPO2}.
 \end{proof}
 
 \begin{remark}
 We stress that 
\begin{equation}\label{MAGA3}
\begin{split}&
{\mbox{the assumption in Theorem~\ref{thm:Hopf}
that $u(\xi)\le u(x)$ for all~$x\in\R^n$}}
\\&{\mbox{cannot be weakened by assuming that~$u(\xi)\le u(x)$ for all~$x\in\overline\Omega$.}}
\end{split}\end{equation}
As a counterexample, let~$u$
be the minimizer of
$$ \int_{-1}^1|\nabla u(x)|^2\, dx+\iint_{\R^2}\frac{|u(x)-u(y)|^2}{|x-y|^{1+2s}}\, d x\, dy$$
among the functions in~$H^1_0(\R)$ such that~$u=u_0$ outside~$(-1,1)$,
for a given smooth function~$u_0$ such
that 
$$
\chi_{(-4,-3)\cup(3,4)}\le
u_0\le
\chi_{(-5,-2)\cup(2,5)}.$$
By the Sobolev Embedding Theorem, we know that~$u$ is continuous in~$\R$ and, in view of
its minimality property, it satisfies~${\mathcal{L}}u=0$ in~$(-1,1)$.
Thus, let~$\xi\in[-1,1]$ be such that~$u(\xi)=\min_{[-1,1]} u$. We claim that
\begin{equation}\label{MAGA}
\xi\in(-1,1).\end{equation} Indeed, if not, we would have that~$u\ge0$ in~$[-1,1]$.
Since~$u\le0$ outside~$(-1,1)$ we know from the weak maximum principle
(see e.g.~\cite[Theorem~1.2]{BDVV}) that~
$$\text{$u\le0$ in~$[-1,1]$},$$ 
and consequently~$u$ vanishes identically in~$[-1,1]$. In particular, the origin would provide an interior maximum for~$u$. Accordingly,
by the strong maximum principle (see e.g.~\cite[Theorem 1.4]{BDVV}),
we gather that~$u$ vanishes identically in~$\R$. Since this is a contradiction with the values of~$u$ in~$(-4,-3)\cup(3,4)$, the proof of~\eqref{MAGA} is completed.
Now, from~\eqref{MAGA}, we deduce that
\begin{equation}\label{MAGA2}
u'(\xi)=0.\end{equation} We also take~$\delta>0$ such that~$\Omega:=(\xi,\xi+\delta)\subset(-1,1)$.
By construction, we have that~$\xi$ is a minimizing point for~$u$ in~$\overline\Omega$ and that~${\mathcal{L}}u=0$ in~$\Omega$. Thus,
equation~\eqref{MAGA2} proves the observation in~\eqref{MAGA3}. 
 \end{remark}

  \section{The principal Dirichlet eigenvalue for $\LL$} \label{sec.firstEig}
  
  In this section we introduce the notion
  of principal Dirichlet ei\-gen\-va\-lue for 
  $\LL$, 
  and we prove some regularity results for the associated
  eigenfunctions. In what follows, we tacitly inherit all the assumptions
  and notation introduced so far; in particular, $s\in (0,1)$ and
  $\Omega\subseteq\R^n$ is a \emph{bounded open set with $C^1$ boundary}.
  \medskip
  
  To begin with, we give the following definition.
  \begin{definition} \label{def.Eigen}
   We
   define the \emph{prin\-ci\-pal
   Di\-ri\-chlet eigenvalue of $\LL$ in $\Omega$} as 
   \begin{equation} \label{eq.defLambdaLL}
    \lambda_{\LL}(\Omega) :=
    \inf\bigg\{\frac{\mathcal{D}_{\Omega,s}(u)}{\|u\|^2_{L^2(\Omega)}}:\,u\in \mathbb{X}(\Omega)
    \setminus\{0\}\bigg\} \in [0,\infty),
   \end{equation}
   where $\mathcal{D}_{\Omega,s}$ is the quadratic form
   defined in \eqref{eq.defquadD}.
  \end{definition}
  \begin{remark} \label{rem.proplambdaLL}
  Before proceeding we list, for a future reference, some simple pro\-per\-ties
  of $\lambda_\LL(\cdot)$ which easily follow from
   its very definition.
   \begin{enumerate}
    \item For every $x_0\in\R^n$, one has
    $$\lambda_{\LL}(\Omega) = \lambda_\LL(x_0+\Omega).$$
    
    \item For every $t > 0$, one has
    \vspace*{0.05cm}
    
    \begin{itemize}
    \item[(i)] 
     $t^{-2s}\lambda_{\LL}(\Omega)
    \leq \lambda_{\LL}(t\Omega) \leq t^{-2}\lambda_{\LL}(\Omega)$
    if $0 < t \leq 1$;
    \vspace*{0.08cm}
    
    \item[(ii)]
    $t^{-2}\lambda_{\LL}(\Omega)
    \leq \lambda_{\LL}(t\Omega) \leq t^{-2s}\lambda_{\LL}(\Omega)$ if $t > 1$.
    \end{itemize}
    \vspace*{0.2cm}
    
    \item Let $\lambda_1(\Omega)$ be
    the principal
    Dirichlet eigenvalue of $-\Delta$ in $\Omega$, i.e.,
    $$\lambda_1(\Omega) := \inf\bigg\{\frac{\int_{\Omega}|\nabla u|^2\, dx}{\|u\|^2_{L^2(\Omega)}}
    :\,u\in H_0^1(\Omega)\setminus\{0\}\bigg\}.$$
    Then, one has the bound
    $$\lambda_{\LL}(\Omega) \geq \lambda_1(\Omega) > 0.$$
    
    \item There exists a positive constant $c = c(\Omega)>0$ such that
    $$\lambda_{\LL}(\Omega) \leq c \, \lambda_1 (\Omega).$$
   \end{enumerate}
  \end{remark}
  Then, by using standard arguments of the Calculus of Variations
  (see, e.g., the approach in \cite[Proposition 9]{SerVal}),
  it is possible to prove the following result.
  
  \begin{theorem} \label{thm.existencelambdaLL}
    The infimum in \eqref{eq.defLambdaLL} is achieved.
    More precisely, there exists \emph{a unique}
     $u_0\in \mathbb{X}(\Omega)\setminus\{0\}$ such that:
     \begin{enumerate}
      \item $\|u_0\|_{L^2(\Omega)} = 1$ and $\lambda_\LL(\Omega) = \mathcal{D}_{\Omega,s}(u_0)$;
      \vspace*{0.05cm}
      
      \item $u_0\geq 0$ a.e.\,in $\Omega$.
     \end{enumerate}
     Furthermore, $u_0$ is a weak solution of the following problem
     \begin{equation} \label{eq.PDEweak}
      \begin{cases} 
      \LL u = \lambda_\LL(\Omega)u & \text{in $\Omega$}, \\
      u \equiv 0 & \text{in $\R^n\setminus\Omega$},
      \end{cases}
      \end{equation}
   namely
     \begin{equation} \label{eq:uzeroPDE} 
     \begin{split}
      \int_{\Omega}\langle \nabla u_0,\nabla v\rangle\, dx
      & + \iint_{\R^{2n}}\frac{(u_0(x)-u_0(y))(v(x)-v(y))}{|x-y|^{n+2s}}\, d x\, dy
     \\
     & 
     =
     \lambda_\LL(\Omega)\int_{\R^n}u_0v\, dx \qquad\quad
     \text{for every $v\in \mathbb{X}(\Omega)$}.
     \end{split}
     \end{equation}
  \end{theorem}
  
  \begin{definition}     We refer to $u_0$ given by Theorem~\ref{thm.existencelambdaLL} as the 
     \emph{principal
  eigenfunction of $\LL$} \emph{(}in $\Omega$\emph{)}.
  \end{definition}
  
  We devote the rest of this section
  to prove some regularity properties of $u_0$. 
  To begin with, we establish the following \emph{global boundedness} result.
  \begin{theorem} \label{thm:uzeroglobalBd}
   The principal eigenfunction $u_0$ is \emph{globally bounded} on $\Omega$.
  \end{theorem}
  The proof of Theorem \ref{thm:uzeroglobalBd} relies on the 
  classical method by Stampacchia, and is essentially analogous to that
  of \cite[Theorem 4.7]{BDVV} (see also \cite{IRE, SVBd}).  
  However, we present it here with all the details
  for the sake of completeness.
  
  \begin{proof}[Proof of Theorem~\ref{thm:uzeroglobalBd}] 
   To begin with, we arbitrarily fix $\delta > 0$ and we set
   $$\widetilde{u}_0 := \sqrt{\delta}\,u_0.$$
   Moreover, for every $k\in\N$, we define $C_k := 1-2^{-k}$ and
	$$v_k:=\widetilde{u}_0-C_k, \quad w_k:=(v_k)_+:=\max\{v_k,0\},\quad
	U_k:= \|w_k\|_{L^{2}(\Omega)}^2.$$ 
  We explicitly point out that, in view of these definitions, one has
  \begin{itemize}
   \item[(a)] $\|\widetilde{u}_0\|^2_{L^2(\Omega)} = \delta\,\|u_0\|^2_{L^2(\Omega)} = \delta$;
   \vspace*{0.05cm}
   \item[(b)] $w_0 = v_0 = \widetilde{u}_0$ (since $C_0 = 0$);
   \vspace*{0.05cm}
   \item[(c)] $v_k\geq v_{k+1}$ and $w_k\geq w_{k+1}$ (since $C_k < C_{k+1}$).
  \end{itemize}
  We now observe that, since $u_0\in\mathbb{X}(\Omega)\subseteq H^1(\R^n)$, we
  have $v_k\in H^1_{\mathrm{loc}}(\R^n)$; fur\-ther\-mo\-re, since $\widetilde{u}_0 = \sqrt{\delta}\,
  u_0\equiv 0$
  a.e.\,in $\R^n\setminus\Omega$, one also has
  $$v_k = \widetilde{u}_0-C_k = -C_k < 0\quad\text{on $\R^n\setminus\Omega$},$$
  and thus $w_k = (v_k)_+\in\mathbb{X}(\Omega)$
  (remind that $\Omega$ is bounded).
  We are then entitled to use the function $w_k$ as a test function
  in \eqref{eq:uzeroPDE}, obtaining
  \begin{equation} \label{eq:PDEsolveduzerotilde}
  \begin{split}
   \int_{\Omega}\langle \nabla \widetilde{u}_0,\nabla w_k\rangle\, dx
   & +  \iint_{\R^{2n}}\frac{(\widetilde{u}_0(x)-\widetilde{u}_0(y))(w_k(x)-w_k(y))}{|x-y|^{n+2s}}\, d x\, dy
   \\
   & = \lambda_\LL(\Omega)\,\int_{\Omega}\widetilde{u}_0w_k\, dx.
   \end{split}
  \end{equation}
  To proceed further we notice that, for a.e.\,$x,y\in\R^n$, one has
  \begin{equation}\label{EAFS2}
	 \begin{split} |w_{k}(x)-w_{k}(y)|^2
	 & = |(v_{k})_+(x)-(v_{k})_+(y)|^2 \\[0.2cm] 
	 & \leq ( (v_{k})_+(x)-(v_{k})_+(y))(v_{k}(x)-v_{k}(y)) \\[0.2cm]
	&= (w_{k}(x)-w_{k}(y))(\widetilde{u}_0 (x)-\widetilde{u}_0(y)).
	\end{split}
	\end{equation}
  Moreover, taking into account the definition of $w_k$, we get
  \begin{equation} \label{eq:localpart}
   \int_\Omega\langle \nabla\widetilde{u}_0,\nabla w_k\rangle \, d x =
	\int_{\Omega\cap\{ \widetilde{u}_0 > C_k\}}
	\langle \nabla\widetilde{u}_0,\nabla v_k\rangle\, dx =
	\int_\Omega|\nabla w_{k}(x)|^2\, dx.
  \end{equation}
  Gathering together \eqref{eq:PDEsolveduzerotilde}, 
  \eqref{EAFS2} and \eqref{eq:localpart}, we obtain
  \begin{equation} \label{eq:estimnablawkI} 
    \int_\Omega|\nabla w_{k}(x)|^2\, dx
   \leq \lambda_\LL(\Omega)\,\int_{\Omega}\widetilde{u}_0w_k\, dx.
  \end{equation}
  We then claim that, for every $k\geq 1$, one has
  \begin{equation} \label{eq:tildeuzeroleqwk}
   \widetilde{u}_0 < 2^{k}w_{k-1}\quad\text{on $\{w_{k} > 0\}$}.
  \end{equation}
  Indeed, if $x\in\{w_k > 0\}$, we have 
  $\widetilde{u}_0(x) > C_{k} > C_{k-1}$; as a consequence,
  \begin{align*}
   2^{k}w_{k-1} (x)& \geq (2^k-1)w_{k-1}(x) = 
   (2^{k}-1)(\widetilde{u}_0(x)-C_{k-1}) \\
   & = \widetilde{u}_0 (x)+ (2^k-2)\big(\widetilde{u}_0(x)-C_k\big) > \widetilde{u}_0(x),
  \end{align*}
  which is exactly the claimed \eqref{eq:tildeuzeroleqwk}. 
  By combining~\eqref{eq:tildeuzeroleqwk} with \eqref{eq:estimnablawkI}, and taking
  into account that $w_k\leq w_{k-1}$ a.e.\,in $\R^n$, for every $k\geq 1$ we get
  \begin{equation} \label{eq:estimnablawkII}
   \begin{split}
    & \int_\Omega|\nabla w_{k}(x)|^2\, dx 
    \leq \lambda_\LL(\Omega)\,\int_{\{w_k > 0\}}\widetilde{u}_0 w_k\, dx \\[0.1cm]
    & \qquad \leq 2^k\cdot\lambda_\LL(\Omega)\,\int_{\{w_k > 0\}}
    w_{k-1}w_k\, dx \\[0.1cm]
    & \qquad \leq 2^k\cdot\lambda_\LL(\Omega)\,\int_{\Omega}w_{k-1}^2\, dx
    = 2^k\cdot\lambda_\LL(\Omega)\,\|w_{k-1}\|^2_{L^2(\Omega)} \\[0.1cm]
    & \qquad = 2^k\cdot\lambda_\LL(\Omega)\,U_{k-1}.
   \end{split}
  \end{equation}
  We now turn to estimate from below the left-hand side of 
  \eqref{eq:estimnablawkII}. To this end we first observe that, owing to the
  very definition of $w_k$, we have
	\begin{equation*}
	\{ w_{k}>0\} = \{\widetilde{u}_0 > C_k\} 
	\subseteq \{w_{k-1}>{2^{-k}}\}\qquad{\mbox{ for all }}k\geq 1.
	\end{equation*}
	As a consequence, we obtain
	\begin{equation} \label{eq:Uklower}
	 \begin{split}
	 U_{k-1} & = \int_{\Omega}w_{k-1}^2\, dx
	 \geq \int_{\{w_{k-1} > 2^{-k}\}}w_{k-1}^2\, dx \\[0.15cm]
	 & \geq 2^{-2k}\,\big|\{w_{k-1} > 2^{-k}\}\big|
	 \geq 2^{-2k}\big|\{w_k > 0\}|.
	 \end{split}
	\end{equation}
	Using the H\"older inequality (with exponents $2^*/2$ and $n/2$, being~$2^*:=\frac{2n}{n-2}$
	the Sobolev exponent), jointly with the
	Sobolev inequality, from 
	\eqref{eq:estimnablawkII}-\eqref{eq:Uklower}
	we obtain the following estimate for every~$k\ge1$
	\begin{equation} \label{eq:finaleperconcludere}
	 \begin{split}
	 U_k &= \|w_k\|^2_{L^2(\Omega)}
	 \leq \bigg(\int_\Omega w_k^{2^*}\, d x\bigg)^{2/{2^*}}\,
	 \big|\{w_k > 0\}\big|^{2/n} \\[0.1cm]
	 & \leq \mathbf{c}_S\,\int_\Omega|\nabla w_k|^2\, dx
	 \cdot
	 \big|\{w_k > 0\}\big|^{2/n} \\[0.1cm]
	 & \leq \mathbf{c}_S\,\big(2^k\,\lambda_\LL(\Omega)U_{k-1}\big)\,\big(
	 2^{2k}U_{k-1}\big)^{2/n} \\[0.2cm]
	 & = \mathbf{c}'\,\big(2^{1+4/n}\big)^{k-1}\,U_{k-1}^{1+2/n}
,
	 \end{split}
	\end{equation}
where~$\mathbf{c}_S$ is the Sobolev constant and
 $\mathbf{c}' := 2^{1+4/n}\,\mathbf{c}_S\,\lambda_\LL(\Omega)$.

With estimate \eqref{eq:finaleperconcludere} at hand, we are finally
	ready to complete the proof. Indeed, since $2/n > 0$ and
	$$\eta := 2^{1+4/n} > 1,$$
	we deduce from \cite[Lemma 7.1]{Giusti} that $U_k\to 0 $ as $k\to\infty$, provided that
	\begin{equation*}
	 U_0 = \|\widetilde{u}_0\|^2_{L^2(\Omega)} = \delta
	< (\mathbf{c}')^{-n/2}\,\eta^{-n^2/4}. 
	\end{equation*}
	As a consequence, if $\delta > 0$ is small enough, we obtain
	$$0 = \lim_{k\to\infty} U_k = 
	\lim_{k\to\infty}\int_\Omega(\widetilde{u}_0-C_k)_+^2\, dx
	= \int_{\Omega}(\widetilde{u}_0-1)_+^2\, dx.$$
	Bearing in mind that $\widetilde{u}_0 = \sqrt{\delta}\,u_0$ (and $u_0\geq 0$), we then get
	$$0\leq u_0\leq \frac{1}{\sqrt{\delta}}\qquad\text{a.e.\,in $\Omega$},$$
	from which we conclude that $u_0\in L^\infty(\Omega)$.
 \end{proof} 
  Now, if $\de\Omega$ is sufficiently
  regular, by com\-bi\-ning Theorem \ref{thm:uzeroglobalBd}
  with Theorems \ref{thm:W2pBL}-\ref{thm:solvabLL} we obtain the next key result.
  
  \begin{theorem} \label{thm:globalCalfauzero}
   Let~$s\in (0,1)$ and $\alpha\in (0,1)$ be such that
   $$\alpha+2s < 1.$$
  Let us suppose that $\de\Omega$ is of class $C^{2,\alpha}$. Denoting
   by $u_0\in\mathbb{X}(\Omega)$ the prin\-ci\-pal
   eigenfunction of $\LL$ in $\Omega$ \emph{(}according
   to Theorem \ref{thm.existencelambdaLL}\emph{)}, we have
   $$u_0\in C_b(\R^n)\cap C^{2,\alpha}(\overline{\Omega}).$$
   Moreover, $u_0$ is a \emph{classical solution} of \eqref{eq.PDEweak}, that is,
   \begin{equation}\label{3.2BIS}
   \begin{cases}
   \LL u_0 = \lambda_\LL(\Omega)u_0 & \text{pointwise in $\Omega$}, \\
   u_0 \equiv 0 & \text{in $\R^n\setminus\Omega$}.
   \end{cases}\end{equation}
   Furthermore, $u_0$ satisfies the following additional properties:
   \begin{enumerate} 
   \item $u_0 > 0$ pointwise in $\Omega$;
   \item denoting by $\nu$ the external unit normal to $\de\Omega$, we have 
   \begin{equation} \label{eq.denuneguzero}
    \de_{\nu}u(\xi) < 0\quad{\mbox{ for all }} \xi\in\de\Omega.
   \end{equation}
   \end{enumerate}
  \end{theorem}
  
  \begin{proof} First of all we observe that, setting $f:= \lambda_\LL(\Omega)\,u_0$,
  from Theorem \ref{thm:uzeroglobalBd} we infer that $f\in L^\infty(\Omega)$; as a consequence,
  since $u_0$ is the (unique) weak solution of
   $$(\mathrm{D}_f)\qquad\quad\begin{cases}
   \LL u = f & \text{in $\Omega$}, \\
   u \equiv 0 &\text{in $\R^n\setminus\Omega$},
   \end{cases}
   $$
   (and $\de\Omega$ is of class $C^{2,\alpha}$), we deduce from Theorem \ref{thm:W2pBL} that 
   $$\text{$u_0\in C^{1,\beta}(\overline{\Omega})$ for
   every $\beta\in (0,1)$}.$$ 
   In particular, $f \in C^\alpha(\overline{\Omega})$. In view of this fact,
   and taking into account our as\-sump\-tions on $s$ and on $\de\Omega$, 
   we can apply 
   Theorem 	 
   \ref{thm:solvabLL} to $u_0$: therefore,
	$$u_0\in C_b(\R^n)\cap C^{2,\alpha}(\overline{\Omega}),$$
   and $u_0$ is a \emph{classical solution} of $(\mathrm{D}_f)$, that is,
   $u_0$ satisfies \eqref{3.2BIS}. 
   To complete the proof, we now turn to prove the validity of
   properties (1)-(2).
   \medskip
   
   \noindent (1)\,\,Since $u_0$ satisfies \eqref{3.2BIS} 
   and $u_0\geq 0$ in $\Omega$, we have
   \begin{equation} \label{eq:solveduzerosign}
   \text{$\LL u_0 = \lambda_\LL(\Omega)u_0\geq 0$ pointwise in $\Omega$ and $u_0\equiv 0$
   on $\R^n\setminus\Omega$};
   \end{equation}
   hence, we can invoke the Strong Maximum Principle for $\LL$ in \cite[Theorem 1.4]{BDVV},
   ensuring that either $u_0\equiv 0$ in $\R^n$ or $u_0 > 0$ in $\Omega$.
   On the other hand, since  $u_0\not\equiv 0$, 
   we conclude that
   $u_0 > 0$ in $\Omega$, as desired.
   \medskip
   
   \noindent (2)\,\,Since 
   $u_0 > 0$ in $\Omega$ and $u_0\equiv 0$ on $\de\Omega$, by \eqref{eq:solveduzerosign} we derive that
   \begin{itemize}
    \item[(a)] $\LL u_0\geq 0$ pointwise in $\Omega$;
    \vspace*{0.05cm}
    
    \item[(b)] $u_0(\xi) = \min_{\de\Omega}u_0 = 0$ for every $\xi\in\de\Omega$.
   \end{itemize}
   As a consequence, taking into account that $u_0\in C^{1}(\overline{\Omega})$
   and $\de\Omega$ is of class $C^{2,\alpha}$, we are entitled to apply
   the Hopf-type Theorem \ref{thm:Hopf}, obtaining that
   $$\de_{\nu}u_0(\xi) < 0\qquad{\mbox{ for all }} \xi\in\de\Omega.$$
   This is precisely the desired \eqref{eq.denuneguzero}, and the proof is complete.
  \end{proof}
  
 We want to stress that one can get almost everywhere positivity of~$u_0$ by means of the strong maximum principle for weak solutions proved in \cite{BSM,BiMuVe}.
\medskip

  Finally, by exploiting Theorem \ref{thm:globalCalfauzero},
  we can prove Lemma \ref{lem:convexity} below. We stress that in the purely local regime, the following lemma has been proved in \cite{SWYY} without any restriction on $\delta$. On the contrary, the same property is not known to hold in the purely fractional case.
  As a matter of fact, convexity and superharmonicity properties
  for fractional eigenfunctions can reserve surprisingly
  severe difficulties, see e.g.~\cite{KASI}.
  
  \begin{lemma} \label{lem:convexity}
Let $s\in (0,1/2)$ 
and let us suppose that $\Omega$ is
uniformly convex, and 
  that $\de\Omega$ is
  of class $C^3$. Then, denoting by $u_0$ the principal eigenfunction of $\LL$ in $\Omega$,
  there exists $\delta_0 > 0$ such that, \emph{for every
  fixed $\delta\in (0,\delta_0)$}, the set
  $$\Omega_\delta := \big\{x\in\Omega:\,u_0(x) > \delta\big\}\quad
  \text{is convex}.$$
  Moreover, $\de\Omega_\delta$ is of class $C^1$.
  \end{lemma}
  \begin{proof}
   We split the proof into three steps.
   \medskip
   
   \textsc{Step I.} In this first step we prove that,
   if $\delta > 0$ is sufficiently small, the su\-per\-le\-vel set
   $\Omega_\delta$ can be realized as a suitable ``deformation'' of $\Omega$.
   To this end we first notice that, owing to the assumptions on $\Omega$,
   Theorem \ref{thm:globalCalfauzero} ensures that
   \begin{itemize}
    \item[(a)] $u_0\in C_b(\R^n)\cap C^{2,\alpha}(\overline{\Omega})$, 
    for every~$\alpha\in (0,1)$ such that~$\alpha+2s < 1$,
     and $u_0$
    is a classical solution of \eqref{eq.PDEweak};
    \vspace*{0.05cm}
    \item[(b)] $u_0 > 0$ pointwise in $\Omega$;
    \vspace*{0.05cm}
    \item[(c)] $\de_\nu u_0(\xi) < 0$ for every $\xi\in\de\Omega$.
   \end{itemize}
   Moreover, since $u_0\equiv 0$ on $\de\Omega$, we have
   $$\text{either $\nabla u_0
   = |\nabla u_0|\nu$ or $\nabla u_0
   = -|\nabla u_0|\nu$ pointwise in $\de\Omega$}.$$
   Now, since property (c)
   implies that $\langle \nabla u_0,\nu\rangle < 0$ on $\de\Omega$, we get
   \begin{equation} \label{eq:nablauzeronu}
    \nabla u_0(\xi) = -|\nabla u_0(\xi)|\nu(\xi)\qquad{\mbox{ for all }} \xi\in\de\Omega;
   \end{equation}
   as a consequence, since $u_0\in C^{2,\alpha}(\overline{\Omega})$, again by 
   property (c) we have
   \begin{equation} \label{eq:bounda}
    0 > -a := \max_{\de\Omega}\de_{\nu}u \geq \langle\nabla u_0(\xi),\nu(\xi)
   \geq - |\nabla u_0(\xi)|\quad{\mbox{ for all }} \xi\in\de\Omega.
   \end{equation}
   Hence, it is possible to find some $d_0 > 0$ such that
   \begin{equation} \label{eq:nablauzerogeq}
    |\nabla u_0|\geq \frac{a}{2}\quad\text{on
   $\Omega_{0} := \big\{x\in\Omega:\,d(x,\de\Omega)< d_0\big\}$}.
   \end{equation}
   Furthermore, bearing in mind that $\de\Omega$ is of class $C^3$, by possibly
   shrinking $d_0$ we can also suppose that
   (see, e.g., \cite[Section~14.6]{GT})
   \begin{itemize}
    \item[(i)] $d(\cdot) := d(\cdot,\de\Omega)$ is of class $C^3$ on $\Omega_0$;
    \item[(ii)] for every $x\in\Omega_0$ there exists a unique $\xi\in\de\Omega$ such that
    $$x-\xi = -d(x)\,\nu(\xi).$$
   \end{itemize}
  For a fixed $\xi\in\de\Omega$, we then consider the function
  $$F(t) := u_0\big(\xi-t\nu(\xi)\big)\qquad {\mbox{ for }} 0\leq t < d_0.$$
  Since $u_0$ is of class $C^{2,\alpha}$ on $\overline{\Omega}$, by 
  \eqref{eq:nablauzeronu}-\eqref{eq:bounda} we have
  \begin{align*}
   F'(t) & = -\langle \nabla u_0\big(\xi-t\nu(\xi)\big),\nu(\xi)\rangle  \\
  &  = -\langle\nabla u_0(\xi),\nu(\xi)\rangle - \big[
  \langle \nabla u_0\big(\xi-t\nu(\xi)\big)-\nabla u_0(\xi),\nu(\xi)\rangle\big] \\
  & = |\nabla u_0(\xi)| - t\|u\|_{C^2(\overline{\Omega})}
  \geq a-t\|u\|_{C^2(\overline{\Omega})};
  \end{align*}
  as a consequence, again by shrinking $d_0$ if necessary, we obtain
  \begin{equation} \label{eq:Fprimelowerbd}
   F'(t)\geq \frac{a}{2}\qquad{\mbox{ for all }}0\leq t < d_0.
  \end{equation}
  In particular, $F$ is \emph{strictly increasing} on $[0,d_0)$, and
  \begin{equation} \label{eq:Fincreasing}
   F(t) = F(t)-F(0) \geq \frac{ta}{2}\qquad{\mbox{ for all }}0\leq t < d_0.
   \end{equation}
  To proceed further, we consider the set
  $\mathcal{O} := \big\{x\in\Omega:\,d(x,\de\Omega) > d_0/2\big\}\subseteq\Omega$
  and we notice that, since $u_0 > 0$ in $\Omega$, we have
  \begin{equation} \label{eq:defczero}
   c_0 := \inf_{\mathcal{O}}u_0 > 0.
   \end{equation}
  Hence, we define
  \begin{equation} \label{eq:choicedeltazero}
   \delta_0 := \frac{1}{2}\,\min\Big\{c_0,\frac{ad_0}{4}\Big\} > 0.
   \end{equation}
  \vspace*{0.05cm}
  
  \noindent Taking into account 
  \eqref{eq:Fincreasing}, it is immediate to see that, for every $\delta\in [0,\delta_0)$, 
  there exists 
  a \emph{unique} point
  $t = t(\xi,\delta)\in [0,d_0)$
  such that 
  $$F\big(t(\xi,\delta)\big) = u_0\big(\xi-t(\xi,\delta)\nu(\xi)\big) = \delta;$$
  then, we set
  \begin{equation*}
   \mathcal{V}_\delta := \big\{\xi-t\nu(\xi):\,\xi\in\de\Omega,\,0\leq t\leq t(\xi,\delta)\big\},\quad
  O_\delta := \Omega\setminus\mathcal{V}_\delta,
  \end{equation*}
  see Figure~\ref{inverse}.
  \begin{center}
  \begin{figure}[ht]
   \includegraphics[scale=0.13]{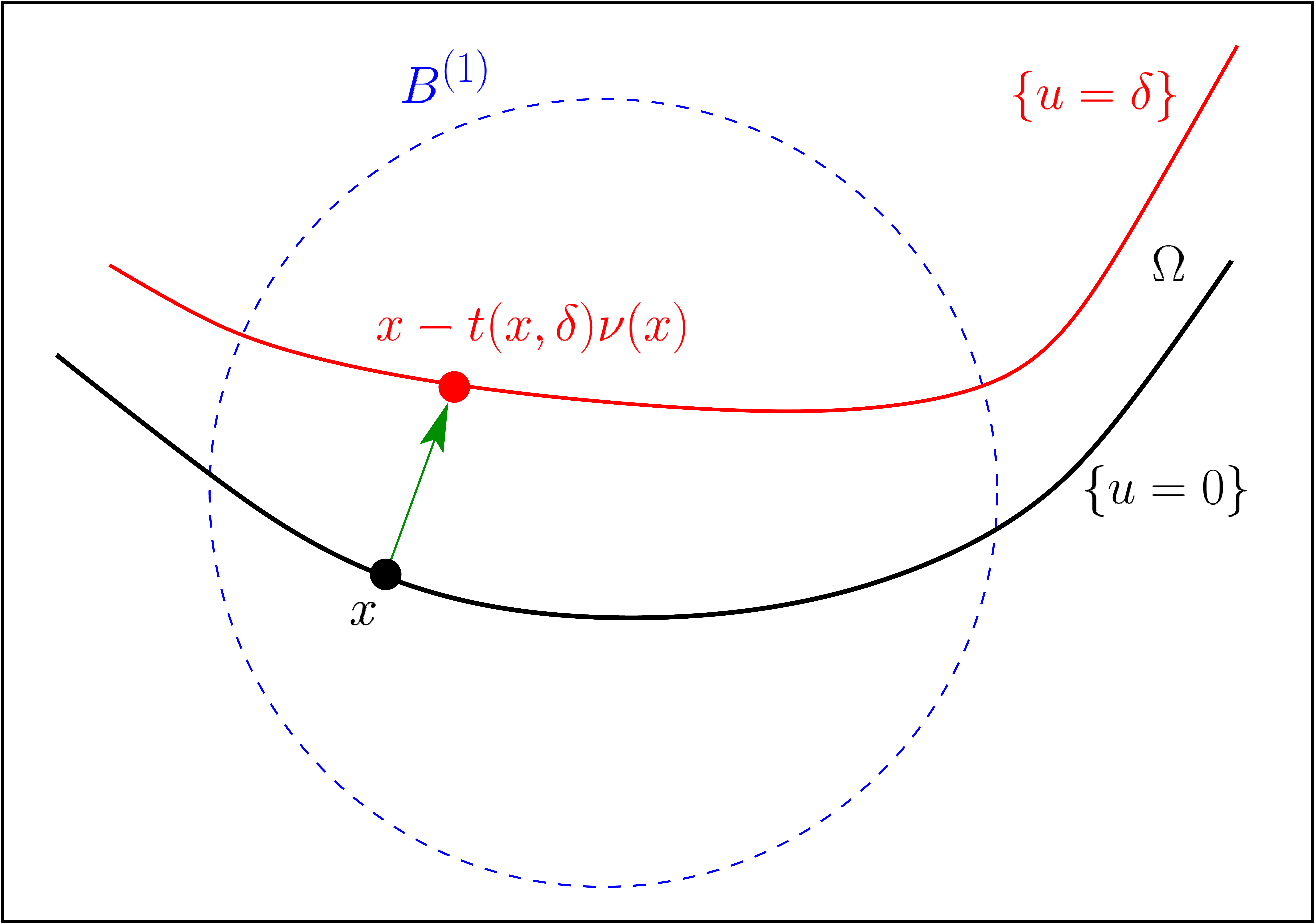}
    \caption{\sl \footnotesize The construction of $O_\delta$} \label{inverse}
   \end{figure}
  \end{center}
  Before proceeding we highlight, for future reference, the following fact: since
  $F(0) = 0$ and $F$ is strictly increasing on $[0,d_0)$, we have
  \begin{equation} \label{eq.txizero}
   t(\xi,0) = 0\quad\text{for every $\xi\in\de\Omega$}.
  \end{equation}
  We now turn to prove that, if $0<\delta<\delta_0$, one has
  \begin{equation} \label{eq:claimOdelta}
   O_\delta = \Omega_\delta = \big\{x\in\Omega:\,u_0(x) > \delta\big\}.
  \end{equation}
  To this end we first observe that, by the very definition of $t(\xi,\delta)$, we have
  $u_0\equiv \delta$ on $\de O_\delta$; moreover, since $F$ is strictly increasing
  on $(0,d_0)$, one also has
  $$u_0\big(\xi-t\nu(\xi)\big) = F(t) \leq F\big(t(\xi,\delta)\big) = \delta
  \qquad{\mbox{ for all }} 0\leq t\leq t(\xi,\delta),$$
  and thus $u_0\leq \delta$ out of $O_\delta$.
  Finally, we show that
  $$u(x) > \delta\quad\text{for every $x\in O_\delta$}.$$
  Let then $x\in O_\delta$ be fixed. If $x\in\mathcal{O}_0$, by 
  \eqref{eq:defczero}-\eqref{eq:choicedeltazero} we have
  $$u_0(x)\geq c_0 > \delta_0 > \delta.$$
  If, instead, $x\notin\mathcal{O}_0$, one has $d(x) = d(x,\de\Omega)\leq d_0/2$, and thus
  $x\in\Omega_0$. By property (ii) of $\Omega_0$, we know that there exists a \emph{unique}
  $\xi\in\de\Omega$ such that 
  $$x = \xi-d(x)\nu(\xi);$$
  on the other hand, since $x\in O_\delta = \Omega\setminus\mathcal{V}_\delta$, we necessarily have
  $$d_0 > d(x) > t(\xi,\delta).$$
  This, together with the strict monotonicity of $F$ on $(0,d_0)$, implies that
  $$u_0(x) = u_0\big(\xi-d(x)\nu(\xi)\big) = F(d(x)) > F(t(\xi,\delta)) = \delta,$$
  and completes the proof of \eqref{eq:claimOdelta}.
  \medskip
  
  \textsc{Step II.} We now turn to prove that, if $\delta_0 > 0$ is as in Step I and
  $\delta\in (0,\delta_0)$, the boundary of $\de\Omega_\delta$ is of class $C^1$.
  To this end we observe that, by crucially exploiting
  identity \eqref{eq:claimOdelta} and the very
  definition of $O_\delta$, one has
  \begin{equation} \label{eq:deOmegaexplicit}
  \de\Omega_\delta = \de O_\delta  = \big\{\xi-t(\xi,\delta)\nu(\xi):\,\xi\in\de\Omega\big\}
  \subseteq \{x\in\Omega:\,d(x) < d_0\}.
  \end{equation}
  This, together with \eqref{eq:nablauzerogeq}, shows that
  $$\text{$|\nabla u_0|\neq 0$ on $\de\Omega_\delta$},$$ 
  and thus
  $\de\Omega_\delta$ is of class $C^1$ (actually, it is of class $C^{2,\alpha}$).
  \medskip
  
  \textsc{Step III.} In this last step we prove that, if $\delta \in (0,\delta_0)$
  (where $\delta_0$ is as in Step~I), the set $\Omega_\delta$ is convex.
  As in Step II, we use in crucial way identity \eqref{eq:claimOdelta}.
  \vspace*{0.05cm}
  
  To begin with, we consider a covering 
  of~$\partial\Omega$
  of finitely many small open balls
  $$\{B^{(i)}\}_{i\in\{1,\dots,N\}}$$
  such that for every~${i\in\{1,\dots,N\}}$
  we can write~$2B^{(i)}\cap\partial\Omega$ as a graph
  in some coordinate direction (where~$2B^{(i)}$
  is the concentric ball of~$B^{(i)}$ with twice the radius of~$B^{(i)}$).
  For~$\delta>0$ sufficiently small, we can
  suppose that
  $$\partial\Omega_\delta
  \subset B^{(1)}\cup\cdots\cup B^{(N)}.$$
  In this setting, it suffices to check that~$B^{(i)}\cap \partial\Omega_\delta$
  can locally be written as a graph of a convex function
  for all~${i\in\{1,\dots,N\}}$. Without loss of generality, we argue for~$i=1$
  and assume that 
  $$2B^{(1)}\cap\partial\Omega = \{x_n=\gamma(x')\}$$
  for some~$\gamma\in C^2(\R^{n-1})$ satisfying
  \begin{equation}\label{KSMGAMMAdue}
  D^2\gamma\ge a_0\,{\rm Id}\qquad(\text{for some $a_0 > 0$}).
 \end{equation} 
  Thus, by \eqref{eq:deOmegaexplicit},
  in this chart~$\partial\Omega_\delta$
 can be locally parameterized by
 \begin{equation}\label{LOCALCHART}
 \begin{split}&(x',\gamma(x'))-t(x',\gamma(x'),\delta)\nu(x',\gamma(x'))\\=\,&
 \Big(
  x'-t(x',\gamma(x'),\delta)\nu'(x',\gamma(x')),\,
  \gamma(x')-t(x',\gamma(x'),\delta)\nu_n(x',\gamma(x'))\Big),
  \end{split}
  \end{equation}
	where we used the notation~$\nu=(\nu',\nu_n)\in\R^{n-1}\times\R$
	and~$x'$ belongs to a domain of~$\R^{n-1}$.
 We now introduce the function
$$(x',t)\longmapsto G(x',\delta,t):=u\big(
(x',\gamma(x'))-t\nu(x',\gamma(x'))\big)-\delta.$$
Let also~$g(x',\delta):=t(x',\gamma(x'),\delta)$.
We notice that, as in~\eqref{eq:Fprimelowerbd}, for~$|t|$ small enough, 
$$ \partial_t G(x',\delta,t)=-
\nabla u\big(
(x',\gamma(x'))-t\nu(x',\gamma(x'))\big)
\cdot\nu(x',\gamma(x'))\ge\frac{a}{2}.
$$
Also~$G$ is locally a~$C^2$ function and~$G\big(x',\delta,g(x',\delta)\big)=0$.
As a consequence,
we find that~$g$ is locally a~$C^2$ function. We also observe that, by~\eqref{LOCALCHART},
the set~$\partial\Omega_\delta$
can be locally parameterized by
\begin{equation}\label{LOCALPAPRAGP}
\Big(
x'-g(x',\delta)\nu'(x',\gamma(x')),\,
\gamma(x')-g(x',\delta)\nu_n(x',\gamma(x'))
\Big).\end{equation}
Now we set
$$ H(x',\delta):=\Big(x'-g(x',\delta)\nu'(x',\gamma(x')),\,\delta\Big).$$
Observe that~$H$ is locally a~$C^2$ function.
Moreover, by \eqref{eq.txizero} we have
\begin{equation}\label{LOCALPAPRAGP-BIS}
g(x',0)=t(x',\gamma(x'),0)=0\end{equation}
and, up to a rotation, we can focus our analysis at a point~$x_0'$
for which
\begin{equation}\label{PIATTO}
\nu(x_0',\gamma(x'_0))=-e_n.
\end{equation}
Therefore
\begin{equation}\label{CORRI}
{\mbox{the Jacobian matrix of~$H$ at~$(x'_0,0)$ is the identity}}
\end{equation}
and we can exploit the Inverse Function Theorem,
denote by~$I$ the inverse function of~$H$ in the vicinity of~$(x'_0,0)$
and have that~$I$ is also a $C^2$ function. Using the notation~$H=(H',H_n)$
and~$I=(I',I_n)$,
we have that~$H(I(y',\varepsilon))=(y',\varepsilon)$ and thus
$$ I_n(y',\varepsilon)=H_n(I(y',\varepsilon))=\varepsilon.$$
Additionally,
$$ x'-g(x',\delta)\nu'(x',\gamma(x'))=H'(I(y',\varepsilon))=y'$$
and 
$$ I'(y',\delta)=
I'\big(x'-g(x',\delta)\nu'(x',\gamma(x')),\delta\big)=
I'(H(x',\delta))=x'.$$
Accordingly, recalling~\eqref{LOCALPAPRAGP},
we can locally parameterize~$\partial\Omega_\delta$
as
\begin{equation*}
 \begin{split}&
 \Big(
  y',\,
  \gamma(I'(y',\delta))-g(I'(y',\delta),\delta)\,\nu_n(I'(y',\delta),\gamma(I'(y',\delta)))
  \Big)\\=\,&
  \Big(
  y',\,
  \gamma(I'(y',\delta))-\eta(y',\delta)\Big),
\end{split}
\end{equation*}
with
$$\eta(y',\delta):=g(I'(y',\delta),\delta)\;\nu_n(I'(y',\delta),\gamma(I'(y',\delta))).$$
For this reason, to complete the proof of the convexity property of~$\partial\Omega_\delta$,
it suffices to check the convexity of the function
 \begin{equation}\label{09876ujhy098765p08098765098765987654} 
  y'\longmapsto \Gamma(y',\delta):=\gamma(I'(y',\delta))-\eta(y',\delta)
 \end{equation} that describes
its graph. To this end, we remark that, for each~$i$, $j\in\{1,\dots,n-1\}$,
\begin{eqnarray*}
\delta_{ij}=\frac{\partial }{\partial y_i'}y'_j=
\frac{\partial }{\partial y_i'}
H_j(I'(y',\varepsilon),\varepsilon)
=\sum_{\ell=1}^{n-1}
\frac{\partial H_j}{\partial x_\ell'}(I'(x',\varepsilon),\varepsilon)\,
\frac{\partial I_\ell}{\partial y_i'}(y',\varepsilon)
\end{eqnarray*}
that is
$$ {\rm Id}=
\frac{\partial H'}{\partial x'}(I'(y',\varepsilon),\varepsilon)\,
\frac{\partial I'}{\partial y'}(y',\varepsilon).$$
Setting~$y'_0:=H'(x'_0,0)$, we thus deduce from~\eqref{CORRI}
that
$$ \frac{\partial I'}{\partial y'}(y'_0,0)={\rm Id}.$$
Also, by~\eqref{PIATTO}, we have that the gradient of~$\gamma$ at~$x'_0$ vanishes.
With these items of information, we find that
\begin{eqnarray*}
\frac{\partial }{\partial y_i'}\gamma(I'(y',\delta))=\sum_{\ell=1}^{n-1}
\frac{\partial \gamma}{\partial x_\ell'}(I'(y',\delta))\,
\frac{\partial I_\ell}{\partial y_i'}(y',\delta)
\end{eqnarray*}
and
\begin{equation*}
\begin{split}\frac{\partial^2}{\partial y_i'\partial y_j'}
\gamma(I'(y',\delta))\Big|_{(y',\delta)=(y'_0,0)}\,&=\,
\sum_{\ell,m=1}^{n-1}
\frac{\partial^2 \gamma}{\partial x_\ell'\partial x'_m}(I'(y'_0,0))\,
\frac{\partial I_\ell}{\partial y_i'}(y'_0,0)
\frac{\partial I_m}{\partial y_j'}(y'_0,0)
\\&=\,\frac{\partial^2\gamma}{\partial y_i'\partial y_j'}(I'(y'_0,0)).
\end{split}
\end{equation*}
Besides, owing to~\eqref{LOCALPAPRAGP-BIS}, we have
$$ \eta(y',0)=g(I'(y',0),0)\;\nu_n(I'(y',0),\gamma(I'(y',0)))
=0$$
and therefore
$$\frac{\partial^2\eta}{\partial y_i'\partial y_j'}(y',0)=0.$$
As a result, recalling the definition of~$\Gamma$ in~\eqref{09876ujhy098765p08098765098765987654},
\begin{equation*}
\frac{\partial^2\Gamma}{\partial y_i'\partial y_j'}(y'_0,0)=\frac{\partial^2}{\partial y_i'\partial y_j'}
\gamma(I'(y',\delta))\Big|_{(y',\delta)=(y'_0,0)}=
\frac{\partial^2\gamma}{\partial y_i'\partial y_j'}(I'(y'_0,0)).
\end{equation*}
By the uniform convexity of the domain~$\Omega$
in~\eqref{KSMGAMMAdue}, we obtain that
$$ \frac{\partial^2\Gamma}{\partial y_i'\partial y_j'}(y'_0,0)\ge a_0\,{\rm Id}.$$
This, together with the fact that $\Gamma$ is a $C^2$ function in~$(y',\delta)$, gives
$$\frac{\partial^2\Gamma}{\partial y_i'\partial y_j'}\ge \frac{a_0}2\,{\rm Id}$$
in a neighborhood of~$(y'_0,0)$,
and proves that~$\Omega_\delta$ is uniformly convex, as desired.
  \end{proof}
  
  \section{A Faber-Krahn inequality and proofs of Theorems~\ref{thm.FK} and \ref{thm:mainFK}}\label{pr:sec}
  This section is devoted to show a \emph{quantitative}
  Faber-Krahn inequality for $\lambda_\LL(\Omega)$ and in particular 
  to prove Theorems~\ref{thm.FK} and~\ref{thm:mainFK}.
   In what follows, $\lambda_\LL(\Omega)$ denotes the principal Dirichlet eigenvalue
  of $\LL$ (in $\Omega$), as given by Definition~\ref{def.Eigen},
  and $u_0$ the corresponding principal eigenfunction
  (according to Theorem \ref{thm.existencelambdaLL}). 
  \medskip
  
  We begin by proving Theorem~\ref{thm.FK}.
  
  \begin{proof}[Proof of Theorem~\ref{thm.FK}]
   Let
   $\widehat{B}^{(m)}$ be the (unique) Euclidean ball with centre $0$ and volume $m$.
   If $u_0\in\mathbb{X}(\Omega)\setminus\{0\}$ is the principal
   eigenfunction of $\LL$ in $\Omega$, we define
   $$u_0^\ast:\R^n\to\R$$
   as the (decreasing) Schwarz symmetrization of $u_0$.
   Now, since $u_0\in\mathbb{X}(\Omega)$, it follows from
   a well-known theorem by Polya-Szeg\"o that
   \begin{equation} \label{eq.PSLoc}
    u_0^\ast\in \mathbb{X}(\widehat{B}^{(m)})\qquad\text{and}\qquad
   \int_{\widehat{B}^{(m)}}|\nabla u_0^\ast|^2\, dx
   \leq \int_{\Omega}|\nabla u|^2\, dx;
   \end{equation}
   furthermore, by \cite[Theorem A.1]{FS} we also have
   \begin{equation} \label{eq.PSNonloc}
    \iint_{\R^{2n}}\frac{|u_0^\ast(x)-u_0^\ast(y)|^2}{|x-y|^{n+2s}}\, d x\, dy
   \leq \iint_{\R^{2n}}\frac{|u_0(x)-u_0(y)|^2}{|x-y|^{n+2s}}\, d x\, dy.
   \end{equation}
   Gathering together all these facts and recalling~(1) in Theorem~\ref{thm.existencelambdaLL}, we then get
   \begin{equation} \label{eq.estimuusharp}
   \begin{split}
    \lambda_{\LL}(\Omega)
    & = \mathcal{D}_{\Omega,s}(u_0)
    = \int_{\Omega}|\nabla u_0|^2\, dx
    +\iint_{\R^{2n}}\frac{|u_0(x)-u_0(y)|^2}{|x-y|^{n+2s}}\, d x\, dy \\
    & \geq 
    \int_{\widehat{B}^{(m)}}|\nabla u_0^\ast|^2\, dx
    + \iint_{\R^{2n}}\frac{|u_0^\ast(x)-u_0^\ast(y)|^2}{|x-y|^{n+2s}}\, d x\, dy
    \\
    & = \DD_{\widehat{B}^{(m)},s}\big(u_0^\ast\big) 
    \geq \lambda_{\LL}(\widehat {B}^{(m)}).
   \end{split}
   \end{equation}
   From this, reminding that $\lambda_{\LL}(\cdot)$ is translation-invariant
   (see Remark \ref{rem.proplambdaLL}-(1)), we derive the validity
   of \eqref{eq.FK} for every Euclidean ball $B^{(m)}$ with volume $m$.
   \medskip
   
   To complete the proof of Theorem~\ref{thm.FK}, let us suppose that 
   $$\lambda_{\LL}(\Omega) = \lambda_{\LL}(B^{(m)})$$
   for some (and hence, for every) ball $B^{(m)}$ with $|B^{(m)}| = m$.
   By \eqref{eq.estimuusharp} we have
   \begin{align*}
    & \int_{\Omega}|\nabla u_0|^2\, dx
    +\iint_{\R^{2n}}\frac{|u_0(x)-u_0(y)|^2}{|x-y|^{n+2s}}\, d x\, dy
    = \lambda_\LL(\Omega) \\
    & \qquad = \lambda_\LL(\widehat{B}^{(m)}) =
    \int_{\widehat{B}^{(m)}}|\nabla (u_0)^\ast|^2\, dx
    + \iint_{\R^{2n}}\frac{|u_0^\ast(x)-u_0^\ast(y)|^2}{|x-y|^{n+2s}}\, d x\, dy.
    \end{align*}
    In particular, by \eqref{eq.PSLoc}-\eqref{eq.PSNonloc} we get
    $$
    \iint_{\R^{2n}}\frac{|u_0(x)-u_0(y)|^2}{|x-y|^{n+2s}}\, d x\, dy 
    = \iint_{\R^{2n}}\frac{|u_0^\ast(x)-u_0^\ast(y)|^2}{|x-y|^{n+2s}}\, d x\, dy.$$
    We are then entitled to apply once again \cite[Theorem A.1]{FS},
    which ensures that
    $u_0$ must be proportional to a translate of a symmetric decreasing function.
    As a consequence of this fact, we immediately deduce that
    $$\Omega = \{x\in\R^n:\,u_0(x) > 0\}$$
    must be a ball (up to a set of zero Lebesgue measure).
  \end{proof}
    
  Now that we have established Theorem \ref{thm.FK}, we can finally prove the main result
  of this paper, namely, the \emph{quantitative version} of inequality \eqref{eq.FK} presented in
  Theorem \ref{thm:mainFK}.
  
  The proof of Theorem \ref{thm:mainFK} is based on the following technical lemma.
  
  \begin{lemma} \label{lem.Melas} Let~$s\in(0,1/2)$.
  Let $\Omega \subset \mathbb{R}^n$ be bounded open
    set with boundary $\partial \Omega$ of class $C^3$.
   Let $m := |\Omega|$, and let
   $B^{(m)}$ be any Euclidean ball with volume $m$.
   Moreover, let~$u_0$ be the principal eigenfunction of~$\LL$ in~$\Omega$,
   let $\delta_0 > 0$ be as in Lemma \ref{lem:convexity}, and let
\begin{equation}\label{NOME}0 < \delta < \min\big\{\tfrac{1}{2}\,m^{-1/2},\delta_0\big\}.\end{equation}
   Let also~$\Omega_\delta:=\{x\in\Omega\,:\, u_0(x)>\delta\}$.
   
   Then, there exists a small $\widetilde\varepsilon > 0$, only depending
   on $n$ and $s$, with the following property: if $0<\varepsilon<\widetilde\varepsilon$
   is such that
   \begin{equation} \label{eq.assLambda}
    \lambda_\LL(\Omega)\leq (1+\varepsilon)\lambda_\LL(B^{(m)}),
   \end{equation}
   then we have the estimate
   \begin{equation*}
    |\Omega_\delta| \geq \left[1-\frac{2n}{s}\cdot\max\{\delta|\Omega|^{1/2},\varepsilon\}\right]\cdot
    |\Omega|.
   \end{equation*}
  \end{lemma}
  
  \begin{proof}
  First of all we observe that, since $\de\Omega$ is of class
  $C^3$, from Theorem~\ref{thm:globalCalfauzero} we de\-ri\-ve that the principal
  eigenfunction $u_0$ of $\LL$ in $\Omega$ satisfies
  $$u_0\in C_b(\R^n)\cap C^{2}(\overline{\Omega})$$
  (actually, $u_0\in C^{2,\alpha}(\overline{\Omega})$ if $\alpha+2s < 1$);
  moreover, since $\delta < \delta_0$, by Lemma \ref{lem:convexity} we know that
  $\de\Omega_\delta$ is of class $C^1$.
  We then consider the function 
  $$v:= (u_0-\delta)_+.$$
  Since $u_0\in \mathbb{X}(\Omega)$ and $u_0\leq \delta$ on $\R^n\setminus\Omega_\delta$,
  it is readily seen that $v\in\mathbb{X}(\Omega_\delta)\subseteq\mathbb{X}(\Omega)$;
  as a consequence, since
  $u_0$ is a weak solution of \eqref{eq.PDEweak}, we have
   \begin{align*}
    \mathcal{D}_{\Omega_\delta,s}(v)
    & = \int_{\Omega_\delta}|\nabla v|^2\, dx
    + \iint_{\R^{2n}}\frac{|v(x)-v(y)|^2}{|x-y|^{n+2s}}\, d x\, dy \\
    & =
    \int_{\Omega}\langle \nabla u_0,\nabla v\rangle\, dx
    + \iint_{\R^{2n}}\frac{(u_0(x)-u_0(y))(v(x)-v(y))}{|x-y|^{n+2s}}\, d x\, dy
    \\
    & = \lambda_{\LL}(\Omega)\int_{\Omega}
    u_0\,v\, dx
    \\&= \lambda_{\LL}(\Omega)\int_{\Omega_\delta}
    (u_0-\delta)\,u_0\, dx \\
    & \leq 
    \lambda_{\LL}(\Omega)\bigg(\int_{\Omega_\delta}
    (u_0-\delta)^2\, dx\bigg)^{1/2}\\&= \lambda_{\LL}(\Omega)\,\|v\|_{L^2(\Omega_\delta)}
   \end{align*}
   where we have used
H\"older's inequality and the fact that~$\|u_0\|_{L^2(\Omega)} = 1$.

   On the other hand, since $u_0\leq \delta$ in $\Omega\setminus\Omega_\delta$
   and $\delta|\Omega|^{1/2}<1/2$ (by the choice of
   $\delta$ and the definition of $\Omega_\delta$), an application
   of Minkowski's inequality gives
   \begin{align*}
    \|v\|_{L^2(\Omega_\delta)}
    & = \bigg(\int_{\Omega_\delta}(u_0-\delta)^2\, dx\bigg)^{1/2}
    \geq
    \bigg(\int_{\Omega_\delta}u_0^2\, dx\bigg)^{1/2}- 
    \bigg(\int_{\Omega_\delta}\delta^2\, dx\bigg)^{1/2} \\
    & \geq \bigg(1-\int_{\Omega\setminus\Omega_\delta}\delta^2\bigg)^{1/2}
    - \delta|\Omega|^{1/2}
    \geq 1-2\delta|\Omega|^{1/2}.
   \end{align*}
   Gathering together all these estimates, 
   and reminding \eqref{eq.assLambda}, we obtain
   \begin{equation} \label{eq.estimlambdaOmegadelta}
   \begin{split}&
    \lambda_{\LL}(\Omega_\delta)\leq \frac{ \mathcal{D}_{\Omega_\delta,s}(v)}{ \|v\|_{L^2(\Omega_\delta)}^2}
     \leq\frac{\lambda_{\LL}(\Omega)}{
    \|v\|_{L^2(\Omega_\delta)}}
    \leq \lambda_{\LL}(\Omega)\big(1-2\delta|\Omega|^{1/2}\big)^{-1} \\
    &\qquad \leq (1+\varepsilon)\lambda_\LL(B^{(m)})\big(1-2\delta|\Omega|^{1/2}\big)^{-1}.
    \end{split}
   \end{equation}
   U\-sing Remark \ref{rem.proplambdaLL}-(2) and the Faber-Krahn inequality
  in Theorem \ref{thm.FK} (notice that we are in the position of applying Theorem \ref{thm.FK} for the set~$t\Omega_\delta$
 in light of the regularity result in Lemma~\ref{lem:convexity}, and notice also that $|t\Omega_\delta| = |\Omega| = m$), we get
  \begin{equation} \label{eq.applicationFKRemark}
   t^{-2s}
   \lambda_\LL(\Omega_\delta)
   \geq \lambda_\LL(t\Omega_\delta) \geq \lambda_\LL(B^{(m)}),
   \qquad\text{where
   $t:= \frac{|\Omega|^{1/n}}{|\Omega_\delta|^{1/n}} > 1$}.
  \end{equation}
  By combining 
  \eqref{eq.estimlambdaOmegadelta} with 
  \eqref{eq.applicationFKRemark}, we then get
  \begin{equation}\label{doeti9b645}
   \frac{|\Omega_\delta|}{|\Omega|}
 \geq \bigg(\frac{\lambda_\LL(B^{(m)})}{\lambda_\LL(\Omega_\delta)}\bigg)^{n/(2s)}
   \geq 
   \bigg[\frac{1-2\delta|\Omega|^{1/2}}{1+\varepsilon}\bigg]^{n/(2s)}.
  \end{equation}
  Finally, if $\widetilde\varepsilon > 0$ is sufficiently small and $0<\varepsilon<\widetilde\varepsilon$,
  we obtain
  $$ \bigg[\frac{1-2\delta|\Omega|^{1/2}}{1+\varepsilon}\bigg]^{n/(2s)} \geq 1-\frac{2n}{s}\cdot\max\{\delta|\Omega|^{1/2},\varepsilon\}.$$
  {F}rom this and~\eqref{doeti9b645}, we obtain the 
desired result.
  \end{proof}

  Now we provide a convexity result that turns out to be useful for the proof
   of Theorem \ref{thm:mainFK}:
\begin{lemma}\label{ADDI}
Let~$\Omega\subset\R^n$ be open, bounded
and convex.
Then, there exists $\widehat\varepsilon > 0$ 
 with the following property: if there exists a ball $B\subseteq \Omega$ such that
\begin{equation}\label{BDE} 
 |B|\ge (1-\varepsilon)|\Omega|\qquad\text{for some $0<\varepsilon<\widehat\varepsilon$},
\end{equation}
then there exists a ball~$B_*$, which is concentric to~$B$,
such that~$\Omega\subseteq B_*$ and
\begin{equation}\label{BDE22}  |\Omega|\ge \big(1-C\varepsilon^{\frac{2}{n+1}}\big)|B_*|.\end{equation}
Here, the positive constant~$C$ depends only on~$n$.
\end{lemma}

\begin{proof} 
 Up to a translation, we can assume that~$B$ is centered at the origin. We assume that~$B$
 has radius~$R$ and we take~$P\in\overline\Omega$
 maximizing the distance to the origin among points in~$\Omega$. Let also~$\delta:=|P|-R$.
 By construction, we have that~$\delta\ge0$ (since~$\Omega$ contains~$B$)
 and that~$\Omega$ is contained in the ball~$B_*$ of radius~$R+\delta$. 
 Since~$|P|=R+\delta$, up to a rotation we can suppose that~$P=(0,\dots,0,R+\delta)$.
 We also consider the convex hull~$H$
 of~$P$ and~$B$. By construction, $H$ lies in the closure of~$\Omega$.
 Let~$K$ be the right circular cone obtained by intersecting~$H$ and the halfspace~$\{x_n\ge R\}$.
 Notice that the height of the cone~$K$ is equal to~$\delta$ and we denote by~$r$ the radius
 of its basis. By triangular similitude (see the triangles~$\stackrel{\triangle}{PTO}$
and~$\stackrel{\triangle}{PRQ}$ in Figure~\ref{HULL} on page~\pageref{HULL}), we see that
 $$ \frac{r}{\delta}=\frac{R}{\sqrt{(R+\delta)^2-R^2}}.$$
 As a consequence,
 \begin{equation}\label{DSFBD} 
  \begin{split}&|\Omega\setminus B|\ge |K|\ge cr^{n-1}\delta=
 c\delta\left(\frac{\delta R}{\sqrt{(R+\delta)^2-R^2}}\right)^{n-1}\\&\qquad=
 c\delta
 \left( \frac{\delta R}{ \sqrt{2\delta R+\delta^2}}\right)^{n-1}=
 c\delta^{\frac{n+1}2}\left( \frac{R}{ \sqrt{2R+\delta}}
 \right)^{n-1},\end{split}
 \end{equation}
for some~$c>0$ depending only on~$n$.

On the other hand, in view of~\eqref{BDE},
\begin{equation*} |\Omega\setminus B|=|\Omega|-|B|\le\varepsilon|\Omega|.\end{equation*}
Combining this and~\eqref{DSFBD}, we have that
\begin{equation}\label{KJS:0olktyhgro} c\delta^{\frac{n+1}2}\left( \frac{R}{ \sqrt{2R+\delta}}
\right)^{n-1}\le\varepsilon|\Omega|.\end{equation}
Also, by~\eqref{BDE} we know that, for~$\varepsilon$ sufficiently small,
$$|\Omega|\le2(1-\varepsilon)|\Omega|\le2|B|=
CR^n$$
for some~$C>0$ depending only on~$n$.
{F}rom this and~\eqref{KJS:0olktyhgro}, up to renaming~$c$ we conclude that
\begin{equation}\label{LABB} \frac{c\delta^{\frac{n+1}2}}{ (2R+\delta)^{\frac{n-1}2}}
\le\varepsilon R.\end{equation}

Now we claim that
\begin{equation}\label{LABB2}
\delta\le \widetilde{C} R,
\end{equation}
where~$\widetilde{C}:=2+\frac{ 2^{\frac{n+1}2}}{c}$. Indeed, 
suppose by contradiction that~$\delta> \widetilde{C} R$.
Then, by~\eqref{LABB},
\begin{eqnarray*}&&
1\ge\varepsilon\ge \frac{c\delta^{\frac{n+1}2}}{ R\,(2R+\delta)^{\frac{n-1}2}}=
\frac{c\delta}{ R\,\left(\frac{2R}\delta+1\right)^{\frac{n-1}2}}\ge
\frac{c\delta}{ R\,\left(\frac{2}{\widetilde{C}}+1\right)^{\frac{n-1}2}}\\&&\qquad\qquad\ge
\frac{c\delta}{ R\,\left(1+1\right)^{\frac{n-1}2}}\ge
\frac{c\,\widetilde C}{ 2^{\frac{n-1}2}}\ge2.
\end{eqnarray*}
This is a contradiction and thus~\eqref{LABB2} is established.

Now, combining~\eqref{LABB} and~\eqref{LABB2} we find that
$$ \frac{\widetilde{c}\,\delta^{\frac{n+1}2}}{ R^{\frac{n-1}2}}
\le\varepsilon R,$$
with~$\widetilde{c}:=\frac{c}{(2+\widetilde{C})^{\frac{n-1}2}}$, and therefore
\begin{equation}\label{UJSNYnboB}
\delta\le C_\star\,\varepsilon^{\frac2{n+1}}\,R
\end{equation}
with~$C_\star>0$ depending only on~$n$.

We also remark that
$$ |B_*|= |B|\frac{(R+\delta)^n}{R^n}\le
|\Omega|\frac{(R+\delta)^n}{R^n}.$$
This and~\eqref{UJSNYnboB} entail that
\begin{equation}\label{UJSNYnboB-2} |B_*|\le\big(1+C_\star\,\varepsilon^{\frac2{n+1}}\big)
|\Omega|.
\end{equation}
It is also useful to point out that, for every~$t\ge0$,
$$ (1+t)(1-t)=1-t^2\le1,$$
thus we deduce from~\eqref{UJSNYnboB-2} that
\begin{equation*}|B_*|\le\frac{\big(1+C_\star\,\varepsilon^{\frac2{n+1}}\big)\big(1-C_\star\,\varepsilon^{\frac2{n+1}}\big)
|\Omega|}{\big(1-C_\star\,\varepsilon^{\frac2{n+1}}\big)}\le
\frac{
|\Omega|}{\big(1-C_\star\,\varepsilon^{\frac2{n+1}}\big)},
\end{equation*}
as desired.
\end{proof}

In spite of some comments appeared in the literature, we believe that the exponent in formula~\eqref{BDE22} of
Lemma~\ref{ADDI}
is optimal, as remarked in Appendix~\ref{AER}. 

\begin{proof}[Proof of Theorem \ref{thm:mainFK}]
  Along the proof, constants depending only on $n$, $s$ and $\Omega$ may change 
  passing from a line to another. Nevertheless, to avoid a cumbersome notation, we will keep the same 
  symbol $C$ for all of them.
  
  Let $u^{\ast}_0$ be the decreasing Schwarz symmetrization of the first eigenfunction $u_0$ (given by Theorem~\ref{thm.existencelambdaLL}). We recall 
  from Theorem \ref{thm.existencelambdaLL} that we can assume $\|u_0\|_{L^{2}}=1$, 
  and hence  $\|u^{\ast}_0\|_{L^{2}}=1$ as well. We define the sets 
  $$\Gamma(t):= \left\{ x \in \Omega: u_0 (x) =t \right\},$$
  and
  $$\Gamma^{\ast}(t):= \left\{ x \in \Omega: u^{\ast}_0 (x) =t \right\}.$$
  We further define $T:= \sup_{\Omega}u_0 \in (0,+\infty)$ and the function
  $$\psi(t) := \int_{\Gamma(t)}\dfrac{1}{|\nabla u_0|}d\mathcal{H}^{n-1}, \quad {\mbox{ for all }}t \in (0,T).$$
  We recall that
  \begin{equation}\label{eq:GammaPsi}
  \mathcal{H}^{n-1}(\Gamma(t)) \leq \psi(t) \int_{\Gamma(t)}|\nabla u_0| \, d\mathcal{H}^{n-1},
  \end{equation}
  and, thanks to the classical isoperimetric inequality, 
  \begin{equation}\label{eq:GammaAstGamma}
  \mathcal{H}^{n-1}(\Gamma^{\ast}(t)) \leq \mathcal{H}^{n-1}(\Gamma(t)).
  \end{equation}
  We also recall that $\Gamma(t)$ is the boundary of the set 
  $\Omega_t := \{x\in \Omega: u_0(x) >t\}$ and $\Gamma^{\ast}(t)$ is the boundary 
  of a ball with volume equal to $|\Omega_t|$. Therefore, 
  \begin{equation}\label{eq:Hn-1Gamma}
  \mathcal{H}^{n-1}(\Gamma^{\ast}(t)) = n \,|B_1|^{1/n} \,|\Omega_t|^{1-1/n}, 
  \quad \textrm{for every } t \in [0,T),
  \end{equation}
   where, as usual, $|B_1|$ denotes the $n$-dimensional Lebesgue measure of the unit ball. \medskip
  
Now we take~$\widetilde\varepsilon$ to be as in Lemma~\ref{lem.Melas} and~$\widehat\varepsilon$
as in Lemma~\ref{ADDI}.
  We also define~$\overline\varepsilon:=\min\big\{\tfrac{1}{4|\Omega|},\delta_0^2\big\}$,
  where~$\delta_0$ is as in Lemma~\ref{lem:convexity}. We stress that~$\widetilde\varepsilon$,
  $\widehat\varepsilon$ and~$\overline\varepsilon$
  are small quantities depending only on the structural parameters of the problem and we suppose that
  the parameter~$\varepsilon_0$ in the statement of Theorem~\ref{thm:mainFK}
  satisfies
  \begin{equation}\label{NOME2}
  \varepsilon_0<\min\left\{\widetilde\varepsilon,{\widehat{\varepsilon}}^{2n}, \overline\varepsilon,|\Omega|,\frac{s^2}{16 n^2|\Omega|}\right\}.
  \end{equation}
  \medskip
  
  \textsc{Step I.} 
  We first prove that, if~$\varepsilon\in(0,\varepsilon_0)$ and \eqref{eq:AssumptionFK} is satisfied, then
  \begin{equation}\label{eq:Step1}
  \int_{0}^{T}\left[ \mathcal{H}^{n-1}(\Gamma(t))^2 - \mathcal{H}^{n-1}(\Gamma^{\ast}(t))^2 \right] \dfrac{1}{\psi(t)}\, d t \leq \lambda_{\LL}(B) \varepsilon.
  \end{equation}
To this aim, we use \eqref{eq:GammaPsi} and~\eqref{eq:GammaAstGamma} to observe that
\begin{eqnarray*}&&
 \int_{\Omega}|\nabla u_0|^2 \, d x =
\int_{0}^{T}\int_{\Gamma(t)}|\nabla u_0| \, d\mathcal{H}^{n-1}\, d t =
\int_{0}^{T}\mathcal{H}^{n-1}(\Gamma(t))^2 \dfrac{d t}{\psi(t)} \\&&\qquad\qquad\ge
\int_{0}^{T}\mathcal{H}^{n-1}(\Gamma^{\ast}(t))^2 \dfrac{d t}{\psi(t)} = 
\int_{B}|\nabla u_0^{\ast}|^2 \,d x.
\end{eqnarray*}
Consequently, using~\eqref{eq:AssumptionFK} and~\eqref{eq.PSNonloc}, we get that
\begin{align*}
 \int_{0}^{T}\bigg[ \mathcal{H}^{n-1}&(\Gamma(t))^2 - \mathcal{H}^{n-1}(\Gamma^{\ast}(t))^2 \bigg]
  \dfrac{1}{\psi(t)}\,d t 
 = \int_{\Omega}|\nabla u_0|^2 \, d x - \int_{B}|\nabla u_0^{\ast}|^2 \, d x
\\[0.1cm]
&  = \lambda_{\LL}(\Omega) - \lambda_{\LL}(B) + 
\iint_{\mathbb{R}^{2n}}\dfrac{|u_0^{\ast}(x)-u_0^{\ast}(y)|^2}{|x-y|^{n+2s}}\, d x \, d y  
 \\[0.1cm]
 & \qquad\qquad - \iint_{\mathbb{R}^{2n}}\dfrac{|u_0(x)-u_0(y)|^2}{|x-y|^{n+2s}}\, d x \, d y \\
& \leq \lambda_{\LL}(\Omega)-\lambda_{\LL}(B) 
\leq \lambda_{\LL}(B) \varepsilon,
\end{align*}
  which is precisely \eqref{eq:Step1}.
  \medskip

  \textsc{Step II.} We prove that
  if~$\varepsilon\in(0,\varepsilon_0)$ and \eqref{eq:AssumptionFK} is satisfied, then
there exists a structural constant $C >0$ and~$\delta>0$ sufficiently small such that
  \begin{equation}\label{eq:Step2}
  \inf_{0\leq t \leq \delta} \left[ \mathcal{H}^{n-1}(\Gamma(t))^2 - \mathcal{H}^{n-1}(\Gamma^{\ast}(t))^2 \right] \leq C \varepsilon^{1/2}.
  \end{equation}
  For this we take \begin{equation}\label{soey5yyvtb}\delta := \varepsilon^{1/2}\end{equation} and we point out that, as a consequence of~\eqref{NOME2}, the condition in~\eqref{NOME} is satisfied. This allows us to use
  Lemma~\ref{lem.Melas}, yielding that
  \begin{equation*}
  |\Omega \setminus \Omega_{\delta}| \le
  \frac{2n\,|\Omega|}{s}\,\max\big\{\varepsilon^{1/2}|\Omega|^{1/2},\varepsilon\big\}=
  \frac{2 n\,|\Omega|^{3/2}\,\varepsilon^{1/2}}{s}
.\end{equation*}   
Therefore, because of the above choice for $\delta$ and exploiting the Cauchy-Schwarz inequality, we also find that
\begin{align*}
\varepsilon = \delta^2 & = \left( \int_{0}^{\delta}d t\right)^2 
\leq \left(\int_{0}^{\delta}\tfrac{d t}{\psi(t)}\right)\,\left( \int_{0}^{\delta}\psi(t)\, d t\right) 
\\[0.15cm]
& = \left(\int_{0}^{\delta}\tfrac{d t}{\psi(t)}\right) \, |\Omega \setminus \Omega_{\delta}|  
 \leq \frac{2 n\,|\Omega|^{3/2}\,\varepsilon^{1/2}}{s}\left(\int_{0}^{\delta}\tfrac{d t}{\psi(t)}\right),
\end{align*}
and thus there exists some $C > 0$ such that
  \begin{equation}\label{eq:StimaPsi}
  \int_{0}^{\delta}\frac{d t}{\psi(t)} \geq C  \varepsilon^{1/2},
  \end{equation}
  Hence, using~\eqref{eq:Step1} and~\eqref{eq:StimaPsi}, 
  and recalling~\eqref{eq:GammaAstGamma}, we deduce that
\begin{equation*}
\begin{aligned}
\inf_{0\leq t \leq \delta}& \left[ \mathcal{H}^{n-1}(\Gamma(t))^2 - 
\mathcal{H}^{n-1}(\Gamma^{\ast}(t))^2 \right]\\ 
&{\leq}\,\inf_{0\leq t \leq \delta}\left[ \mathcal{H}^{n-1}(\Gamma(t))^2 - 
 \mathcal{H}^{n-1}(\Gamma^{\ast}(t))^2 \right] \, C \varepsilon^{-1/2} 
 \int_{0}^{\delta}\dfrac{d t}{\psi(t)} \\
&\leq   \,
C \varepsilon^{-1/2} \int_{0}^{\delta}\left[ \mathcal{H}^{n-1}(\Gamma(t))^2 - \mathcal{H}^{n-1}
(\Gamma^{\ast}(t))^2 \right]\dfrac{1}{\psi(t)}\, d t\\
&{\leq}\, C \varepsilon^{1/2} \lambda_{\LL}(B) ,
\end{aligned}
\end{equation*}  
which gives \eqref{eq:Step2}.  

  We notice that, by the very definition of infimum, and recalling the 
  choice of~$\delta$ in~\eqref{soey5yyvtb}, there exists 
  $\tau \in [0,\delta]=[0,\varepsilon^{1/2}]$ such that   
  \begin{equation}\label{eq:DefInf}
   \mathcal{H}^{n-1}(\Gamma(\tau))^2 - \mathcal{H}^{n-1}(\Gamma^{\ast}(\tau))^2 \leq 2C\varepsilon^{1/2},  
   \end{equation}
   for~$\varepsilon\in(0,\varepsilon_0)$ such that \eqref{eq:AssumptionFK} is satisfied.
   \medskip
  
   \textsc{Step III.}  We let~$\varepsilon\in(0,\varepsilon_0)$ such 
   that \eqref{eq:AssumptionFK} is satisfied, and
   we take~$\tau \in [0,\delta]$ such that \eqref{eq:DefInf} holds true. We prove that
   there exists a structural constant $C >0$ such that
  \begin{equation}\label{eq:Step3}
  \mathcal{H}^{n-1}(\partial \Omega_{\tau}) = \mathcal{H}^{n-1}(\Gamma(\tau)) 
  \leq n |B_1|^{1/n} |\Omega_{\tau}|^{1-1/n} + C  \varepsilon^{1/2}.
\end{equation}     
For this, 
recalling that \eqref{eq:GammaAstGamma} and \eqref{eq:Hn-1Gamma} hold for every $t \in [0,T)$,
and using~\eqref{eq:DefInf}, we have that
\begin{align*}
 & \mathcal{H}^{n-1}(\Gamma(\tau)) - \mathcal{H}^{n-1}(\Gamma^{\ast}(\tau)) =\,
 \dfrac{\mathcal{H}^{n-1}(\Gamma(\tau))^2 - \mathcal{H}^{n-1}(\Gamma^{\ast}(\tau))^2 }{\mathcal{H}^{n-1}
  (\Gamma(\tau)) + \mathcal{H}^{n-1}(\Gamma^{\ast}(\tau))} \\ 
&  \qquad\quad \leq \,\dfrac{2\, C  \varepsilon^{1/2}}{\mathcal{H}^{n-1}(\Gamma(\tau)) + \mathcal{H}^{n-1} 
 (\Gamma^{\ast}(\tau))} \leq\, 
 \dfrac{2\,\varepsilon^{1/2}}{\mathcal{H}^{n-1}(\Gamma^{\ast}(\tau))} \\
 & \qquad\quad
{=}\, \dfrac{2\,
 \varepsilon^{1/2}}{n |B_1|^{1/n}|\Omega_{\tau}|^{1-1/n}}.
\end{align*}
 {F}rom this, eventually modifying the constant $C>0$, we further obtain that
 \begin{equation}\label{dj3857485709999}
 \mathcal{H}^{n-1}(\Gamma(\tau)) \leq n\, 
 |B_1|^{1/n}\,|\Omega_{\tau}|^{1-1/n} + \dfrac{C \, \varepsilon^{1/2}}{|\Omega_{\tau}|^{1-1/n}}.
 \end{equation}
 Furthermore, since~$\tau\le\delta=\varepsilon^{1/2}$, recalling~\eqref{NOME2}, we have that the
 condition in~\eqref{NOME} is satisfied, and therefore
  we can exploit Lemma \ref{lem.Melas}. In this way, we obtain that
  $$|\Omega_{\tau}|^{1-1/n} \geq
  \left[1-\frac{2n}{s}\cdot \varepsilon^{1/2}|\Omega|^{1/2}\right]^{1-1/n}\,    |\Omega|^{1-1/n}\ge 
   \left(\dfrac{|\Omega|}{2}\right)^{1-1/n},$$
  thanks to~\eqref{NOME2}. Plugging this information 
  into~\eqref{dj3857485709999}, we obtain~\eqref{eq:Step3}.
 \medskip 
  
  \textsc{Step IV.} We are now ready to finish the proof of \eqref{eq:Goal1}. 
   Once this is established, the existence of a ball $B_2$ for which \eqref{eq:Goal2} 
   holds follows from the convexity of $\Omega$
  and Lemma~\ref{ADDI} (notice that we are in the position of 
  exploiting Lemma~\ref{ADDI} thanks to~\eqref{NOME2}). \medskip
  
  Hence we focus on the proof of~\eqref{eq:Goal1}.
  We let $\rho>0$ be the inradius of $\Omega_\tau$ and let us 
  consider the ball $B^{(1)}$ whose radius is $\rho$
  and such that $B^{(1)} \subseteq \Omega_\tau$.
  We recall that the convexity of the set~$\Omega_\tau$ ensures that
  the following Bonnesen-type inequality holds,
  see e.g.~\cite{DING, HADW, OSSE}:
  \begin{equation}\label{eq:Bonnesen}
  \left( \dfrac{\mathcal{H}^{n-1}(\partial \Omega_\tau)}{\mathcal{H}^{
  n-1}(\partial B^{(1)})}\right)^{n/(n-1)} - \dfrac{|\Omega_{\tau}|}{|B^{(1)}|} 
  \geq \left[ \left(\dfrac{\mathcal{H}^{n-1}(\partial \Omega_\tau)}{
  \mathcal{H}^{n-1}(\partial B^{(1)})}\right)^{1/(n-1)} -1 \right]^{n}.
  \end{equation}
  Also, from~\eqref{eq:Step3},
  we have that
  \begin{equation*}
 \mathcal{H}^{n-1}(\partial \Omega_\tau) \leq n\,|B_1|^{1/n}\, |\Omega_\tau|^{
 1-1/n}+C\,
  \varepsilon^{1/2}.
    \end{equation*}
  Therefore,
  \begin{equation}\label{eq:Step3,5}
  \begin{split}
  & \left( \dfrac{\mathcal{H}^{n-1}(\partial \Omega_\tau)}
  {\mathcal{H}^{n-1}(\partial B^{(1)})}\right)^{n/(n-1)} - \dfrac{|\Omega_{\tau}|}{|B^{(1)}|} 
  \le \\
  & \qquad\qquad\left( \dfrac{n \, |B_1|^{1/n}\, |\Omega_\tau|^{
 1-1/n}+C\,
  \varepsilon^{1/2}}{\mathcal{H}^{n-1}(\partial B^{(1)})} \right)^{n/(n-1)}- 
  \dfrac{|\Omega_{\tau}|}{|B^{(1)}|} .
  \end{split}
  \end{equation}

Furthermore, using the isoperimetric inequality we see that
\begin{equation*}
\frac{|\Omega_\tau|^{1-1/n}}{{\mathcal{H}}^{n-1}(\partial\Omega_\tau)}\le  
   \frac{|B_1|^{1-1/n}}{{\mathcal{H}}^{n-1}(\partial B_1)}=
   \frac{|B_1|^{1-1/n}}{n\,|B_1|}=\frac{1}{n\,|B_1|^{1/n}}.
\end{equation*}
As a result, we obtain that
  \begin{equation}\label{eq:Step3,75}\begin{split}&
 \left[ \left(\dfrac{\mathcal{H}^{n-1}(\partial \Omega_\tau)}
 {\mathcal{H}^{n-1}(\partial B^{(1)})}\right)^{1/(n-1)} -1 \right]^{n} \ge  
 \left[ \left(\dfrac{| \Omega_
 \tau|^{1-1/n}}{| B^{(1)}|^{1-1/n}}\right)^{1/(n-1)} -1 \right]^{n}\\
 &\qquad\qquad = \left[ \left(\dfrac{| \Omega_\tau|}{| B^{(1)}|}\right)^{1/n} -1 \right]^{n}=
 \dfrac{\left[| \Omega_\tau|^{1/n}-|B^{(1)}|^{1/n}\right]^n }{ | B^{(1)}| } .
 \end{split} \end{equation}
  Now, combining \eqref{eq:Bonnesen}, \eqref{eq:Step3,5} and \eqref{eq:Step3,75}, we find that
  \begin{equation*}
   \left( \dfrac{n \, |B_1|^{1/n}\, |\Omega_\tau|^{
 1-1/n}+C\,
  \varepsilon^{1/2}}{\mathcal{H}^{n-1}(\partial B^{(1)})} \right)^{n/(n-1)}- 
  \dfrac{|\Omega_{\tau}|}{|B^{(1)}|}\ge
  \dfrac{\left[| \Omega_\tau|^{1/n}-|B^{(1)}|^{1/n}\right]^n }{ | B^{(1)}| }
  .\end{equation*}
Recalling that~$B^{(1)}$ has radius~$\rho$, hence $|B^{(1)}|=|B_1|\rho^n$ and
$$\mathcal{H}^{n-1}(\partial B^{(1)})=
\mathcal{H}^{n-1}(\partial B_1)\rho^{n-1}
=n|B_1|\rho^{n-1},$$ 
it thus follows that
  \begin{equation}\label{eq:Step3,8}
  \left(  |\Omega_\tau|^{
 1-1/n}+C\,\varepsilon^{1/2} \right)^{n/(n-1)}- |\Omega_{\tau}|\ge
\left[| \Omega_\tau|^{1/n}-|B^{(1)}|^{1/n}\right]^n 
  .\end{equation}
  We also point out that
\begin{equation}\label{eq:Omega_Tau}
\dfrac{|\Omega|}{2}	\leq |\Omega_\tau| \leq |\Omega|,
\end{equation}  
 for sufficiently small $\tau$ and therefore
  \begin{eqnarray*}
&&  \left(  |\Omega_\tau|^{
 1-1/n}+C\,\varepsilon^{1/2} \right)^{n/(n-1)}
  =|\Omega_\tau|\left(  1+\frac{C\,\varepsilon^{1/2}}{|\Omega_\tau|^{
 1-1/n}} \right)^{n/(n-1)}\\&&\qquad\le
 |\Omega_\tau|\left(  1+\frac{2^{1-1/n}C\,\varepsilon^{1/2}}{|\Omega|^{
 1-1/n}} \right)^{n/(n-1)}\le|\Omega_\tau|\left(  1+C\varepsilon^{1/2}\right),
  \end{eqnarray*}
  up to renaming~$C$ in the last inequality.
  
  Combining this and~\eqref{eq:Step3,8} we gather that
$$ \left[| \Omega_\tau|^{1/n}-|B^{(1)}|^{1/n}\right]^n \le C\,|\Omega_\tau| \varepsilon^{1/2}.$$
This, the Bernoulli inequality and~\eqref{eq:Omega_Tau} entail that
$$ | \Omega_\tau|^{1/n}-|B^{(1)}|^{1/n}\ \le C\,|\Omega_\tau|^{1/n} \varepsilon^{1/(2n)},$$
which gives \eqref{eq:Goal1}.  
 \end{proof}
\begin{appendix}

\section{Convex sets, remarks on the literature, and optimality of Lemma~\ref{ADDI}}\label{AER}

 In the literature, it seems to be suggested (see e.g. the
 end of Section~2 in~\cite{Melas}) that the result in Lemma~\ref{ADDI} could be improved,
 for instance by posing the following natural question:
 
\begin{problem}\label{MER}
Let~$\varepsilon>0$. If~$\Omega\subset\R^n$ is bounded
and convex, contains the unit ball~$B_1$ and
\begin{equation}\label{B1hy} |B_1|\ge (1-\varepsilon)|\Omega|,\end{equation}
is it true that there exists a ball~$B_*$ such that~$\Omega\subseteq B_*$ and
\begin{equation}\label{B2hy} |\Omega|\ge (1-C\varepsilon)|B_*|\end{equation}
for some constant~$C>0$?\end{problem}

Here we show that the answer to Problem~\ref{MER} is {\em negative}.
We provide {\em two counterexamples}, one closely related to the proof of Lemma~\ref{ADDI},
and one in which we additionally assume that the set~$\Omega$
is uniformly convex and its boundary is of class~$C^\infty$.
\medskip

\noindent{\bf Counterexample 1.}
Let~$n\ge2$. Let~$\varepsilon>0$ be small and
$$ \delta:=\varepsilon^{\frac2{ n+1}}.$$
Let~$P:=(0,\dots,0,1+\delta)$
and~$\Omega$ be the convex hull of~$B_1\cup P$, see Figure~\ref{HULL}.

  \begin{center}
  \begin{figure}[ht]
   \includegraphics[scale=0.33]{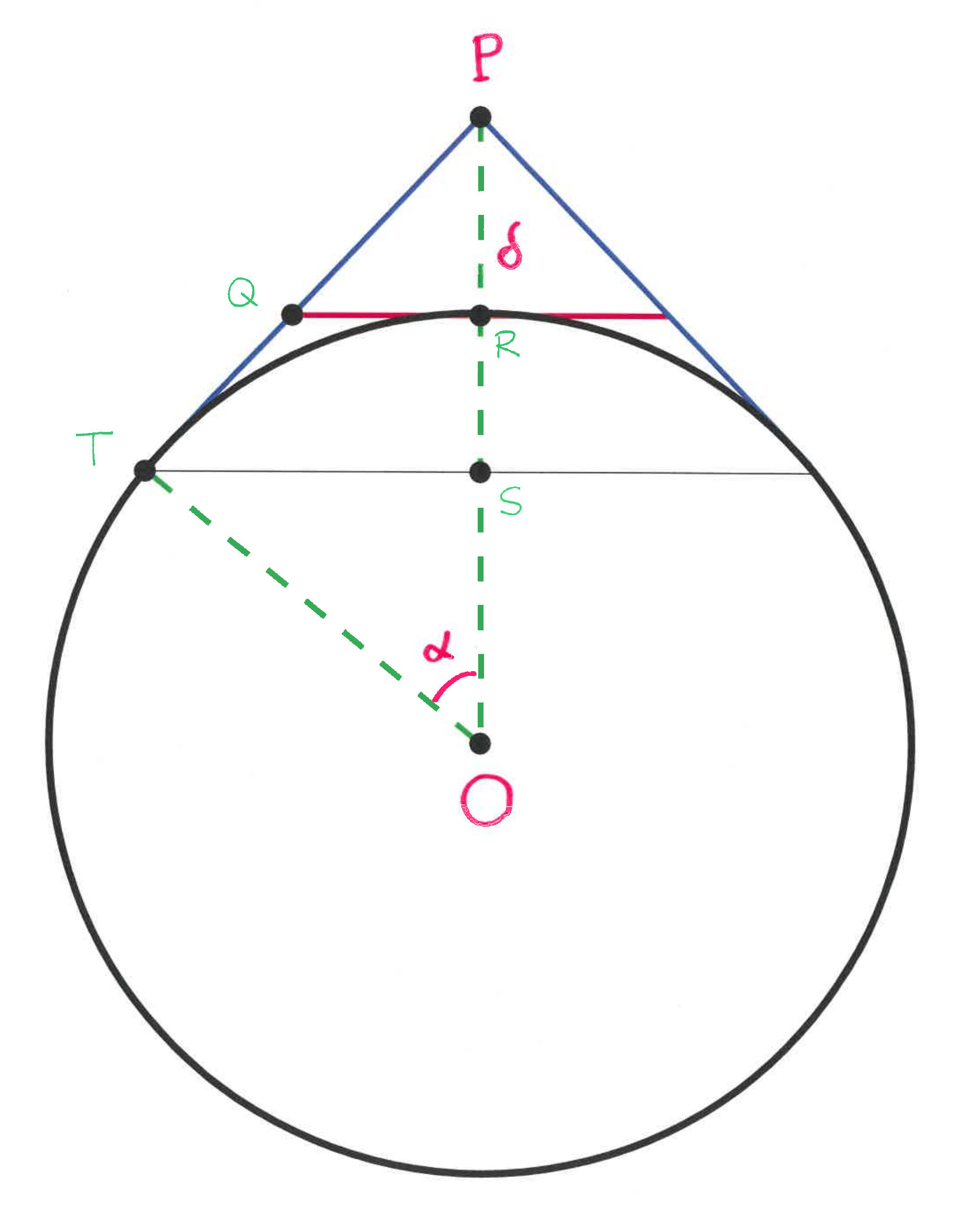}
    \caption{\sl \footnotesize The convex hull of~$B_1\cup P$.} \label{HULL}
   \end{figure}
  \end{center}

We claim that~\eqref{B1hy} holds true.
Indeed, considering Figure~\ref{HULL}, we have 
that~$\overline{OR}=1=\overline{OT}$ and~$\overline{PR}=\delta$.
As a result, we see that~$\overline{OP}=\overline{OR}+\overline{PR}=1+\delta$ and
$$ \overline{PT}=\sqrt{ \overline{OP}^2-\overline{RO}^2}=
\sqrt{ (1+\delta)^2-1}=\sqrt{2\delta+\delta^2}\in[\sqrt\delta,2\sqrt\delta].$$
We also remark that the triangles~$\stackrel{\triangle}{PTO}$,
and~$\stackrel{\triangle}{PST}$ are similar and accordingly
$$ \frac{\overline{ST}}{\overline{PT}}=\frac{\overline{OT}}{\overline{OP}}\qquad{\mbox{and}}\qquad
\frac{\overline{PS}}{\overline{PT}}=\frac{\overline{PT}}{\overline{OP}}
.$$
These identities entail that
\begin{eqnarray*}
&&
{\overline{ST}}=\frac{ {\overline{OT}}\;\; {\overline{PT}} }{\overline{OP}}\le\frac{2{\sqrt\delta}}{1+
 \delta}\le4{\sqrt\delta}\\
{\mbox{and }}&&
{\overline{PS}}=\frac{\overline{PT}^2}{\overline{OP}}=\frac{2\delta+\delta^2}{1+\delta}
\le4\delta.
\end{eqnarray*}
As a result, if
$$ {\mathcal{Q}}:=P+[-4{\sqrt\delta},4{\sqrt\delta}]^{n-1}\times[-4\delta,4\delta],$$
we find that~$\Omega\setminus B_1\subseteq{\mathcal{Q}}$.

{F}rom this, it follows that
\begin{eqnarray*}
&&|\Omega|\le|B_1|+|\Omega\setminus B_1|\le|B_1|+|{\mathcal{Q}}|=
|B_1|+(8{\sqrt\delta})^{n-1}(8\delta)\\&&\qquad\qquad\le |B_1|(1+C\delta^{\frac{n+1}2})=|B_1|(1+C\varepsilon)
\end{eqnarray*}
for some~$C>0$, which gives~\eqref{B1hy} (up to renaming~$\varepsilon$).

Now, consider a ball~$B_*$ such that~$\Omega\subseteq B_*$.
Since~
$$\text{$(0,\dots,0,1+\delta)\in\overline\Omega$ and~$(0,\dots,0,-1)\in\overline\Omega$},$$ 
we have that the diameter of~$\Omega$
is at least~$2+\delta$, hence the radius of~$B_*$ is at least~$1+\frac\delta2$ and thus, using~\eqref{B1hy},
\begin{equation*}
\begin{split}&
|B_*|\ge |B_1|\left(1+\frac\delta2\right)^n\ge |B_1| \left(1+\frac\delta2\right)\ge
\left(1+\frac\delta2\right)\left(1-\varepsilon\right)|\Omega|\\&\qquad=
\left(1+\frac{\delta}2\right)\left(1-\delta^{\frac{n+1}{2}}\right)|\Omega|=
\left(1+\frac{\delta}2+O\left(\delta^{\frac{n+1}{2}}\right)\right)|\Omega|\\&\qquad=
\left(1+\frac{\delta}2+O\left(\delta^{\frac{3}{2}}\right)\right)|\Omega|\ge
\left(1+\frac{\delta}4\right)|\Omega|.
\end{split}
\end{equation*}
This yields that~\eqref{B2hy} is not satisfied in this case, since otherwise
\begin{eqnarray*}
&&1=\frac{|\Omega|}{|\Omega|}\ge\frac{\big(1-\varepsilon\big)|B_*|}{|B_*|/\left(1+\frac{\delta}4\right)}=
\left(1+\frac{\delta}4\right)\left(1-\delta^{\frac{n+1}{2}}\right)\\&&\qquad\qquad=
\left(1+\frac{\delta}4\right) \left( 1+O\left( \delta^{\frac{3}{2}} \right)\right)=1+\frac{\delta}4+O\left( \delta^{\frac{3}{2}} \right)\ge
1+\frac{\delta}8>1,
\end{eqnarray*}
which is a contradiction.\bigskip

\noindent{\bf Counterexample 2.}
For simplicity, we take here~$n=2$ (the case~$n>2$ can be obtained by rotations of the example constructed
in a given plane). Let~$\varepsilon>0$ be small and
$$ \delta:=\varepsilon^{\frac2{ 3}}.$$
We construct here a counterexample of Problem~\ref{MER}
even under the additional assumption that~$\Omega$ is uniformly convex with boundary of class~$C^\infty$. For this,
let~$f\in C^\infty_0([-1,1],\,[0,1])$ with~$f(0)=1$ and
$$ (-\pi,\pi] \ni\vartheta\mapsto g(\vartheta):=1+c\delta f\left(\frac\vartheta{\sqrt\delta}\right),$$
with
$$ c:=\frac{1}{4(1+\|f\|_{C^2(\R)})}.$$
We consider the set
$$ \Omega:=\Big\{
(\rho\cos\vartheta,\rho\sin\vartheta),\;\,
\vartheta\in(-\pi,\pi],\;\, \rho\in[0,g(\vartheta))
\Big\}.$$
We observe that~$g(\vartheta)\in[1,1+\delta]$ for all~$\vartheta\in(-\pi,\pi]$, hence~$B_1\subseteq\Omega$,
and
\begin{equation}\label{WHE}
{\mbox{$g(\vartheta)=1$ whenever~$|\vartheta|>\sqrt\delta$.}}\end{equation}

We let~$\kappa(\vartheta)$ be the curvature of~$\partial\Omega$ at the point~$
(g(\vartheta)\cos\vartheta,\,g(\vartheta)\sin\vartheta)$. Then,
\begin{eqnarray*}
\kappa(\vartheta)&=&
\frac{2\dot{g}^2(\vartheta)-g(\vartheta) \,\ddot{g}(\vartheta)+ g^2(\vartheta)}{\Big(\dot{g}^2(\vartheta)+g^2(\vartheta)\Big)^{\frac{3}{2}}}\\&=&
\frac{2c^2\delta \dot{f}^2\left(\frac\vartheta{\sqrt\delta}\right)
-c g(\vartheta) \,
\ddot{f}\left(\frac\vartheta{\sqrt\delta}\right)+ g^2(\vartheta)}{\left(
c^2\delta \dot{f}^2\left(\frac\vartheta{\sqrt\delta}\right)
+g^2(\vartheta)\right)^{\frac{3}{2}}}.
\end{eqnarray*}
Hence, if
$$ \xi(\vartheta):=c g(\vartheta) \,
\ddot{f}\left(\frac\vartheta{\sqrt\delta}\right),$$
we have that
\begin{equation}\label{4LAXI} |\xi(\vartheta)|\le c (1+\delta) \,\|f\|_{C^2(\R)}\le
2c \,\|f\|_{C^2(\R)}\le\frac12\end{equation}
and
\begin{eqnarray*}
\kappa(\vartheta)&=&
\frac{O(\delta)
-\xi(\vartheta)+ (1+O(\delta))^2}{\left(
O(\delta)+(1+O(\delta))^2\right)^{\frac{3}{2}}}\\
&=&
\frac{1
-\xi(\vartheta)+ O(\delta)}{\left(1+
O(\delta)\right)^{\frac{3}{2}}}\\&=&1
-\xi(\vartheta)+ O(\delta).
\end{eqnarray*}
{F}rom this relation and~\eqref{4LAXI}, we conclude that
$$ \kappa(\vartheta)\in\left[\frac14,2\right],$$
hence~$\Omega$ is uniformly convex.

Moreover, using polar coordinates and~\eqref{WHE},
\begin{eqnarray*}
|\Omega\setminus B_1|&=&\int_{-\pi}^\pi \left[\int_1^{g(\vartheta)}\rho\, d\rho \right]\,d\vartheta
\\&\le& \int_{-\sqrt\delta}^{\sqrt\delta} \left[\int_1^{1+\delta} 2\, d\rho \right]\,d\vartheta\\&\le&
4\,\delta^{\frac32}\\&=&4\,\varepsilon,
\end{eqnarray*}
and therefore
$$ |B_1|\ge|\Omega|-4\varepsilon=\left(1-\frac{4\varepsilon}{|\Omega|}\right)|\Omega|,$$
which shows that~\eqref{B1hy} holds true (up to renaming~$\varepsilon$).

Now, consider a ball~$B_*$ such that~$\Omega\subseteq B_*$.
Since
$$(1+c\delta,0)=
(g(0)\cos 0, g(0)\sin0)\in\overline\Omega$$
and~$(-1,0)=
(g(\pi)\cos \pi, g(\pi)\sin\pi)\in\overline\Omega$, we have that the diameter of~$\Omega$
is at least~$2+c\delta$, hence the radius of~$B_*$ is at least~$1+\frac{c\delta}2$.
Thus, using~\eqref{B1hy},
\begin{equation*}
\begin{split}&
|B_*|\ge |B_1|\left(1+\frac{c\delta}2\right)^n\ge |B_1| \left(1+\frac{c\delta}2\right)\ge
\left(1+\frac{c\delta}2\right)\left(1-\varepsilon\right)|\Omega|\\&\qquad=
\left(1+\frac{c\delta}2\right)\left(1-\delta^{\frac{n+1}{2}}\right)|\Omega|=
\left(1+\frac{c\delta}2+O\left(\delta^{\frac{n+1}{2}}\right)\right)|\Omega|\\&\qquad=
\left(1+\frac{c\delta}2+O\left(\delta^{\frac{3}{2}}\right)\right)|\Omega|\ge
\left(1+\frac{c\delta}4\right)|\Omega|.
\end{split}
\end{equation*}
This yields that~\eqref{B2hy} is not satisfied in this case, since otherwise
\begin{eqnarray*}
&&1=\frac{|\Omega|}{|\Omega|}\ge\frac{\big(1-\varepsilon\big)|B_*|}{|B_*|/\left(1+\frac{c\delta}4\right)}=
\left(1+\frac{c\delta}4\right)\left(1-\delta^{\frac{n+1}{2}}\right)\\&&\qquad\qquad=
\left(1+\frac{c\delta}4\right) \left( 1+O\left( \delta^{\frac{3}{2}} \right)\right)=1+\frac{c\delta}4+O\left( \delta^{\frac{3}{2}} \right)\ge
1+\frac{c\delta}8>1,
\end{eqnarray*}
which is a contradiction.

\section{$C^{2,\alpha}$-regularity for $s< 1/2$}\label{AppC2alpha}
For completeness, we present here
an explicit proof of Theorem \ref{thm:solvabLL}
in the case $s\in (0,1/2)$. In this situation
the action of the fractional operator is better behaved since
it does not produce boundary singularities on functions
that are smooth (or even just Lipschitz) up to the boundary
and with zero external datum. This fact makes
the proof technically easier since it allows one to ``reabsorb''
the fractional operator into the source term of the equation.
For this reason, we thought it could be of some interest, at least
for some readers, to find here a self-contained
result with its own proof.
The precise statement is the following:

 \begin{theorem} 
   Let~$s\in (0,1/2)$ and~$\alpha\in (0,1)$ be such that
  \begin{equation} \label{eq.assumptionas}
    \alpha+2s < 1.
   \end{equation}
Suppose that $\de\Omega$ is of class $C^{2,\alpha}$.
  If $f\in C^{\alpha}(\overline{\Omega})$ and if $u_f\in\mathbb{X}(\Omega)$
  denotes the unique weak solution of $(\mathrm{D})_{f}$
  \emph{(}according to Theorem \ref{thm:existenceBasic}\emph{)}, then
  $$u_f\in C_b(\R^n)\cap C^{2,\alpha}(\overline{\Omega}).$$
  In particular, $u_f$ is a \emph{classical solution} of  $(\mathrm{D})_{f}$.
 \end{theorem}
 \begin{proof}
 We split the proof into different steps.
 \medskip
 
 \textsc{Step I}. We consider the functions space $\mathbb{B}(\Omega)$ defined as follows:
 $$\mathbb{B}(\Omega) := 
  \{u\in C(\R^n):\,\text{$u\equiv 0$ in $\R^n\setminus\Omega$
  and $u|_{\overline{\Omega}}\in C^{2,\alpha}(\overline{\Omega})$}\}
  \subseteq C_b(\R^n).$$
  Then, we claim that there exists
  a constant $\mathbf{c} = \mathbf{c}_{n,s,\alpha,\Omega} > 0$ such that
  \begin{equation} \label{eq.claimCaseI}
   \|(-\Delta)^s u\|_{C^\alpha(\overline{\Omega})}
   \leq \mathbf{c}\|u\|_{C^{1}(\overline{\Omega})}
   <\infty\qquad\forall\,\,u\in \mathbb{B}(\Omega).
  \end{equation}
  In fact, since $u\in \mathbb{B}(\Omega)$ and since, by assumption
  \eqref{eq.assumptionas}, $\beta:=\alpha+2s\in (0,1)$, it 
  is not difficult to recognize that
  $u\in C^{\beta}(\R^n)$, and 
  $$[u]_{C^\beta(\R^n)}
  \leq \mathrm{diam}(\Omega)^{1-\beta} \|u\|_{C^1(\overline{\Omega})}.$$
  As a consequence, since one obviously has
  $\beta = \alpha+2s > 2s$, we are entitled to apply the re\-sult
  in \cite[Prop.\,2.1.7]{SILV}: this gives
  $(-\Delta)^s u\in 
   C^{\alpha}(\R^n)$ and
   \begin{equation} \label{eq.estimseminormDeltas}
   [(-\Delta)^s u]_{C^\alpha(\R^n)}
   \leq c[u]_{C^\beta(\R^n)}
   \leq c\,\mathrm{diam}(\Omega)^{1-\beta}\,\|u\|_{C^1(\overline{\Omega})},
  \end{equation}
  where $c = c_{n,s,\alpha} > 0$ is a constant independent of $u$.
  To complete the proof of \eqref{eq.claimCaseI}, we then turn 
  to estimate the $L^\infty$-norm of $(-\Delta)^s u$
  in terms of
  $\|u\|_{C^{1}(\overline{\Omega})}$. 
  
  First of all we observe that, on account of 
  \eqref{eq.estimseminormDeltas}, for every $x\in\overline{\Omega}$ one has
  \begin{equation} \label{eq.estimsupDeltas}
  \begin{split}
   |(-\Delta)^s u(x)| & \leq [(-\Delta)^s u]_{C^\alpha(\R^n)}\cdot|x-x_0|^\alpha
   + |(-\Delta)^s u(x_0)| \\
   & \leq c\,\mathrm{diam}(\Omega)^{1+\alpha-\beta}\,\|u\|_{C^{1}(\overline{\Omega})}
   + |(-\Delta)^s u(x_0)| \\
   & (\text{setting $\rho := c\,\mathrm{diam}(\Omega)^{1+\alpha-\beta}$}) \\
   & = \rho\,\|u\|_{C^{1}(\overline{\Omega})}
   + |(-\Delta)^s u(x_0)|,
   \end{split}
  \end{equation}
  where $x_0\in\Omega$ is chosen in such a way that
  $$d_0 := \mathrm{dist}(x_0,\de\Omega) = \sup_{x\in\Omega}\big(\mathrm{dist}(x,\de\Omega)\big).$$
  On the other hand, since $u\in\mathbb{B}(\Omega)$
  and $s<1/2$,
  we have the estimate
  \begin{equation} \label{eq.estimDeltasxi}
  \begin{split}
	& |(-\Delta)^s u(x_0)| \leq \int_{\R^n}\frac{|u(x_0)-u(y)|}{|x_0-y|^{n+2s}}\, d y
	= \int_{\R^n}\frac{|u(x_0)-u(x_0-z)|}{|z|^{n+2s}}\, d z
	\\
	& \qquad \leq \int_{\{|z|\leq d_0\}}\frac{1}{|z|^{n+2s}}\,
	|\langle\nabla u(x_0-\tau z),z\rangle|\, d z
	\\
	& \qquad\qquad + 2\|u\|_{L^\infty(\R^n)}\int_{\{|z|>d_0\}}
	\frac{1}{|z|^{n+2s}}\, d z \\
	& \qquad
	\leq \sqrt{n}\|u\|_{C^1(\overline{\Omega})}
	\int_{\{|z|\leq d_0\}}\frac{1}{|z|^{n+2s-1}}\, d z
	\\
	& \qquad\qquad + 2\|u\|_{L^\infty(\R^n)}\int_{\{|z|>d_0\}}
	\frac{1}{|z|^{n+2s}}\, d z \\[0.2cm]
	& \qquad
	= \kappa(d_0^{1-2s}+d_0^{-2s})\|u\|_{C^1(\overline{\Omega})},
  \end{split}
  \end{equation}
  where $\kappa = \kappa_{n,s} > 0$ is another constant which does not depend on
  $u$. Gathering to\-ge\-ther \eqref{eq.estimseminormDeltas}, 
  \eqref{eq.estimsupDeltas} and 
  \eqref{eq.estimDeltasxi}, we finally obtain
  \begin{align*}
   \|(-\Delta)^s u\|_{C^\alpha(\overline{\Omega})}
   & = \|(-\Delta)^s u\|_{L^\infty(\Omega)}
   + [(-\Delta)^s u]_{C^\alpha(\overline{\Omega})}
   \\
   & \leq \big(c\,\mathrm{diam}(\Omega)^{1-\beta}
   + \rho\big)\|u\|_{C^{1}(\overline{\Omega})} \\
   & \qquad\quad
   + \kappa\big(d_0^{2-2s}+d_0^{-2s}\big)\|u\|_{C^1(\overline{\Omega})} \\
   & \leq \mathbf{c}\,\|u\|_{C^1(\overline{\Omega})}
  \end{align*}
  which is exactly the claimed \eqref{eq.claimCaseI}. We explicitly point out that
  the constant $\mathbf{c}$ only depends on $n,s,\alpha$ and $\Omega$ (as the same is true of 
  $c$).
 \medskip
 
 \textsc{Step II}. 
 In this second step, we establish the following facts:
 \begin{enumerate}
   \item $\LL u\in C^{\alpha}(\overline{\Omega})$ for every $u\in\mathbb{B}(\Omega)$;
   \vspace*{0.02cm}
   
   \item there exists a constant $C = C(n,\alpha,s) > 0$ such that
  \begin{equation} \label{eq.schauderLL}
   \|u\|_{C^{2,\alpha}(\overline{\Omega})}
   \leq C\,\big(\|\LL u\|_{C^\alpha(\overline{\Omega})}
   + \sup_{\Omega}|u|\big)\qquad
   \text{for all $u\in\mathbb{B}(\Omega)$}.
  \end{equation}
  \end{enumerate}
  As regards assertion (1), it is a direct consequence
  of \eqref{eq.claimCaseI}, together with the fact that
  $\Delta u\in C^\alpha(\overline{\Omega})$ if 
  $u\in\mathbb{B}(\Omega)\subseteq C^{2,\alpha}(\overline{\Omega})$.
  We now turn to prove assertion (2).
  
  To this end we first notice that, since $\de\Omega$ is of class $C^{2,\alpha}$,
  we are entitled to apply \cite[Thm.\,6.14]{GT}: 
  there exists a constant $C = C(n,\alpha) > 0$
  such that
  \begin{equation} \label{eq.schauderLap}
   \|u\|_{C^{2,\alpha}(\overline{\Omega})}
   \leq C\,\big(\|\Delta u\|_{C^\alpha(\overline{\Omega})}
   + \sup_{\Omega}|u|\big),
  \end{equation}
  for every function $u\in C^{2,\alpha}(\overline{\Omega})$ satisfying
  $u\equiv 0$ on $\de\Omega$.
  Then, by combining
  \eqref{eq.claimCaseI} with \eqref{eq.schauderLap},
  we obtain the following chain of inequalities:
  \begin{equation} \label{eq.penultimaEstimCaseI}
  \begin{split}
   \|u\|_{C^{2,\alpha}(\overline{\Omega})} 
   & 
   \leq C\,\big(\|\Delta u\|_{C^\alpha(\overline{\Omega})}
   + \sup_{\Omega}|u|\big) \\
   & = C\,\big(\|\LL u - (-\Delta)^s u\|_{C^\alpha(\overline{\Omega})}
   + \sup_{\Omega}|u|\big) \\
   & \leq C\,\big(\|\LL u\|_{C^\alpha(\overline{\Omega})}
   + \|(-\Delta)^s u\|_{C^\alpha(\overline{\Omega})}
   + \sup_{\Omega}|u|\big) \\
   & \leq C'\,\big(\|\LL u\|_{C^\alpha(\overline{\Omega})}
   + \|u\|_{C^{1}(\overline{\Omega})}
   + \sup_{\Omega}|u|\big),
  \end{split}
  \end{equation}
  holding true for every $u\in \mathbb{B}(\Omega)$.
  Now, owing again to the regularity of $\de\Omega$,
  we can invoke the
  \emph{global} interpolation
  inequality contained in, e.g., \cite[Chap.\,6]{GT}: there exists a constant
  $\theta > 0$, independent of $u$, such that
  \begin{equation*}
   \|u\|_{C^{1}(\overline{\Omega})}
   \leq \frac{1}{2C'}\|u\|_{C^{2,\alpha}(\overline{\Omega})}
  + \theta\sup_{\Omega}|u|
  \end{equation*}
  This, together with \eqref{eq.penultimaEstimCaseI},
  immediately gives the desired \eqref{eq.schauderLL}.
  \medskip
  
  \textsc{Step III}. In this last step, we complete the proof of the theorem
  by using the so-called \emph{method of continuity}
  (see, e.g., \cite[Thm.\,5.2]{GT}).
  To this end, we first notice that $\mathbb{B}(\Omega)$ is endowed
  with a structure of Banach space by the norm
  $$\|u\|_{\mathbb{B}(\Omega)} := \|u\|_{C^{2,\alpha}(\overline{\Omega})}\qquad
  \forall\,\,u\in\mathbb{B}(\Omega).$$
  Moreover, for every $0\leq t\leq 1$ we define
  $$\LL_t := (1-t)\Delta + t\LL
  = \Delta + t(-\Delta)^s.$$
  Owing to (1)-(2) in Step II, we derive that
  $\LL_t$ maps $\mathbb{B}(\Omega)$ into $C^\alpha(\overline{\Omega})$, and
  \begin{equation} \label{eq.estimGrezza}
  \begin{split}
  \|u\|_{\mathbb{B}(\Omega)}
  & = \|u\|_{C^{2,\alpha}(\overline{\Omega})}
  \leq C\,\big(\|\LL_t u\|_{C^{\alpha}(\overline{\Omega})}
  + \sup_{\Omega}|u|\big),\quad\forall\,\,u\in\mathbb{B}(\Omega),
  \end{split}
  \end{equation}
  where $C > 0$ is a suitable constant \emph{independent of $u$ and $t$}.
  On the other hand, by carefully scrutinizing the proof of
  \cite[Thm.\,4.7]{BDVV}, it is easy to see that
  \begin{equation} \label{eq.estimSupu}
   \sup_{\Omega}|u| = \sup_{\R^n}|u| \leq \kappa\,\|\LL_t u\|_{L^\infty(\Omega)}
   \leq \kappa\,\|\LL_t u\|_{C^{\alpha}(\overline{\Omega})},
  \end{equation}
  where $\kappa > 0$ is another constant independent of $u$ and $t$.
  Thanks to \eqref{eq.estimGrezza}-\eqref{eq.estimSupu}, we are then entitled
  to 
  apply the method of continuity in this setting: indeed,
  \begin{itemize}
  \item $\mathbb{B}(\Omega)$ and $C^{\alpha}(\overline{\Omega})$ are Banach spaces;
  \item $\LL_0,\,\LL_1$ are linear and bounded from $\mathbb{B}(\Omega)$ into 
  $C^{\alpha}(\overline{\Omega})$ (see \eqref{eq.claimCaseI});
  \item there exists a constant $\hat{C} > 0$ such that
  \begin{equation*}
   \|u\|_{\mathbb{B}(\Omega)}\leq \hat{C}\|\LL_t\|_{C^{\alpha}(\overline{\Omega})}
   \qquad\text{for every $u\in\mathbb{B}(\Omega)$ and $t\in[0,1]$},
   \end{equation*}
  \end{itemize}
  As a consequence, since
  $\LL_0 = \Delta$ is surjective, we deduce that also
   $\LL_1 = \LL$  is \emph{surjective}:
   for every $f\in C^\alpha(\overline{\Omega})$ there exists a (unique)
   $\hat{u}_f\in\mathbb{B}(\Omega)$ such that
   \begin{equation} \label{eq:LLhatufpointwise}
    \LL \hat{u}_f = f\quad\text{pointwise in $\Omega$}.
   \end{equation}
  We explicitly notice that, since $\hat{u}_f\in\mathbb{B}(\Omega)$, one has
  $\hat{u}_f\in C_b(\R^n)\cap C^{2,\alpha}(\overline{\Omega})$ and
  $\hat{u}_f\equiv 0$ on $\R^n\setminus\Omega$; thus, by 
  \eqref{eq:LLhatufpointwise} we derive that
  $\hat{u}_f$ is a \emph{classical solution} of 
  $(\mathrm{D})_{f}$.
  In view of these facts, to complete the proof we are left 
  to show that 
  $$\text{$\hat{u}_f = u_f$ a.e.\,in
  $\R^n$}.$$ 
  To this end we first notice that, since $\hat{u}_f\in C^{2,\alpha}(\overline{\Omega})$
  and since $\hat{u}_f\equiv 0$ on $\de\Omega$, we surely have
  $\hat{u}_f\in H_0^1(\Omega)$; this, together with \eqref{eq.indetifyXHzero}, implies that
  \begin{equation} \label{eq:hatufXOmega}
   \hat{u}_f\in\mathbb{X}(\Omega).
   \end{equation}
  Since $\hat{u}_f$ is a classical solution of 
  $(\mathrm{D})_{f}$, from \eqref{eq:hatufXOmega} and Remark \ref{rem:weakclassicalcfr} we
  infer that $\hat{u}_f$ is also a \emph{weak solution} of problem $(\mathrm{D})_{f}$;
  on the other hand, 
  since $u_f$ is the \emph{unique} weak solution this problem, we
  conclude that
  $\hat{u}_f \equiv u_f$ a.e.\,in $\R^n$, as desired.
 \end{proof}
 
 \textbf{Acknowledgments.} We would like to thank Professor
 Lorenzo Brasco for pointing 
 to our attention the paper \cite{Avila} and for some interesting and pleasant
 discussions.
\end{appendix}

\vfill
\end{document}